\documentclass[english,,reqno]{amsart}

\usepackage[english]{babel}
\usepackage[T1]{fontenc}
\usepackage{geometry}
\usepackage{verbatim}
\usepackage{amsthm}
\usepackage{amssymb}
\usepackage{setspace}
\usepackage{enumerate}
\usepackage{fancyhdr}
\usepackage{enumitem}
\usepackage{xcolor}
\usepackage{csquotes}
\usepackage{hyperref}

\usepackage[backend=biber,sorting=nyt,style=ieee]{biblatex}
\numberwithin{equation}{section}

\newcommand{\bb}[1]{\mathbb{#1}}

\theoremstyle{plain}
\newtheorem{thm}{Theorem}[section]
\newtheorem*{thm*}{Theorem}
\newtheorem{lem}[thm]{Lemma}
\newtheorem{prop}[thm]{Proposition}
\newtheorem{cor}[thm]{Corollary}

\theoremstyle{definition}
\newtheorem{defn}{Definition}[section]

\newtheorem{exam}{Example}[section]

\theoremstyle{remark}
\newtheorem*{rem}{Remark}

\begin{document}

\title{Existence and uniqueness of Remotely Almost Periodic solutions of differential equations with piecewise constant argument}

\author[Diego Jaur\'e, Christopher Maul\'en, Manuel Pinto]
{Diego Jaur\'e, Christopher Maul\'en}

\address{Diego Jaur\'e \newline
Facultad de Ingenier\'ia, Universidad de la Rep\'ublica Oriental del Uruguay}
\email{jaure.diego@gmail.com}

\address{Christopher Maul\'en \newline
Departamento de Ingeniería Matemática, Facultad de Ciencias Físicas y Matemáticas, Universidad de Concepción, Concepción, Chile.}
\email{chrismaulen@udec.cl}

\maketitle

{\centering Dedicated to the memory of Manuel Abelardo Pinto Jiménez (1952-2024).\par}

\begin{abstract}
We study differential equations with piecewise constant argument (DEPCA) and establish the existence and uniqueness of remotely almost periodic (RAP) solutions for
\[
x'(t)=A(t)x(t)+B(t)x([t])+f(t).
\]
Under an exponential dichotomy for the associated linear hybrid system \(x'(t)=A(t)x(t)+B(t)x([t])\) and suitable RAP/Lipschitz assumptions on the data, we derive sufficient conditions guaranteeing a unique RAP solution. We further consider perturbed DEPCA of the form
\[
\begin{aligned}
x'(t)&=A(t)x(t)+B(t)x([t])+f(t)+\nu\,g_{\nu}\bigl(t,x(t),x([t])\bigr),\\
y'(t)&=\tilde f\bigl(t,y(t),y([t])\bigr)+\nu\,g_{\nu}\bigl(t,y(t),y([t])\bigr),
\end{aligned}
\]
and prove the existence (and, when appropriate, uniqueness) of RAP solutions for \(\nu\) in a suitable range, under mild uniform Lipschitz and smallness conditions on \(g_{\nu}\). As an application, we obtain RAP solutions for nonautonomous Lasota-Wazewska type models with piecewise constant argument, and show the existence of a unique positive RAP solution under biologically meaningful hypotheses.
\end{abstract}

\maketitle

\section{Introduction}

The study of differential equations with piecewise constant argument (or DEPCAs, for their acronym in English) was initiated by Wiener \cite{288,289,wiener_book}, Myshkis \cite{203}, Cooke and Wiener \cite{51}, and Shah and Wiener \cite{244}.  
The study of DEPCAs is of great applied interest, since a delay provides a mathematical model for a physical or biological system in which the rate of change depends on its past history (see \cite{liming_dai,asymtotyc,cauchy,kuo,robledo}).

A differential equation with piecewise constant argument is a differential equation with a discontinuous argument of the following type:  
\begin{eqnarray}
x'(t) = f(t, x(t), x([t])), \label{investigarnolinear}
\end{eqnarray}  
where $f: \mathbb{R} \times \mathbb{R}^q \times \mathbb{R}^q \rightarrow \mathbb{R}^q$ is continuous on $\mathbb{R}^q$, and $[\cdot]$ denotes the floor (integer part) function. Note that the equation is non-autonomous.  

DEPCAs are hybrid differential equations, that is, they exhibit the structure of continuous dynamical systems over the intervals $]n, n+1[,\ n \in \mathbb{Z}$, and of discrete dynamical systems over $\mathbb{Z}$, combining the properties of both differential and difference equations.

In what follows, we denote by $|\cdot|$ the Euclidean norm.

The following definition is the natural extension of the concept of solution adapted to a DEPCA.

\begin{defn}\label{def_sol_DEPCA}
A function $x:\mathbb{R}\rightarrow\mathbb{R}^{q}$ is a solution of \eqref{investigarnolinear} if the following conditions are satisfied:  
\begin{enumerate}
\item[$(i)$] $x$ is continuous on $\mathbb{R}$;
\item[$(ii)$] the derivative $x'$ of $x$ exists on $\mathbb{R}$  
except possibly at points $t=n$, $n\in\mathbb{Z}$, where the one-sided derivatives exist;
\item[$(iii)$] $x$ satisfies \eqref{investigarnolinear}  
on the interval $(n, n+1)$, for all $n\in\mathbb{Z}$.
\end{enumerate}
\end{defn}

The main objective of this paper is to establish the existence and uniqueness of \textit{Remotely Almost Periodic (RAP)} solutions for DEPCAs, specifically we consider the equations
\begin{eqnarray}
z'(t)&=& A(t)z(t) \label{DEPCA_A}\\
y'(t)&=& A(t)y(t)+B(t)y([t]) \label{DEPCA_lineal}\\
x'(t)&=& A(t)x(t)+B(t)x([t])+f(t), \label{DEPCA_lineal+f}
\end{eqnarray}
where $A, B:\mathbb{R}\rightarrow M_{q\times q}(\mathbb{R})$ are continuous matrices and $f:\mathbb{R}\rightarrow \mathbb{R}^q$ is a bounded uniformly continuous function. In \cite{zhang_depca}, the authors seek remotely almost periodic solutions for \eqref{DEPCA_lineal+f} with constant matrices $A$ and $B$.  Additionally, in \cite{yuan}, the existence of almost periodic solutions for \eqref{DEPCA_lineal+f} is studied.  
In this article, we will address the case where $A$ and $B$ are remotely almost periodic matrices.

It is well known that Sarason defined the notion of a scalar-valued remotely almost periodic function in \cite{sarason}. The class of vector-valued remotely almost periodic functions defined on $\mathbb{R}^n$ was introduced by Yang and Zhang in \cite{yang_zhang_RAP}, where the authors have provided several applications in the study of existence and uniqueness of remotely almost periodic solutions for parabolic boundary value problems (for some results about parabolic boundary value problems, one may refer to \cite{kamenskii_2020,kamenskii_2021,li_yang} and references cited therein). In Propositions 2.4--2.6 in \cite{xu_li_guo}, the authors have examined the existence and uniqueness of remotely almost periodic solutions of multidimensional heat equations, while the main results of Section 3 are concerned with the existence and uniqueness of remotely almost periodic type solutions of certain types of parabolic boundary value problems (see also \cite{yang_zhang_cauchy,yang_zhang_slowly}, where the authors have investigated almost periodic type solutions and slowly oscillating type solutions for various classes of parabolic Cauchy inverse problems). Concerning applications of remotely almost periodic functions, research articles \cite{zhang_piao} by Zhang and Piao, where the authors have investigated the time remotely almost periodic viscosity solutions of Hamilton--Jacobi equations, and \cite{zhang_depca} by Zhang and Jiang, where the authors have investigated remotely almost periodic solutions for a class of systems of differential equations with piecewise constant argument, should be mentioned; see \cite{wiener_book} and the research articles \cite{chiu_pinto_periodic}, 
\cite{coronel_dichotomies}, \cite{kostic_velinov}, 
 and \cite{torres_castillo} for more details about the subject.

To better understand the space of Remotely Almost Periodic functions, denoted by $RAP({\mathbb R } : {\mathbb C }^{q})$, we will recall the notion of a Slowly Oscilating function (the corresponding space is denoted by $SO({\mathbb R } : {\mathbb C }^{q})$ henceforth):
A uniformly continuous function $f:\mathbb{R}\rightarrow \mathbb{C}^q$ is called {\it slowly oscillating} if and only if for every $a\in {\mathbb R}$ we have that
\[
\lim_{|t|\rightarrow +\infty}\| f(t+a)-f(t)\|=0,
\]
where $\|\cdot  \|$ be a fixed norm in $\mathbb C^q$.

Now, we recall the notion of a remotely almost periodic function:

\begin{defn}[\cite{pinto_remotely, zhang_depca}]\label{rap}
A uniformly continuous function $f:\mathbb{R}\rightarrow \mathbb{C}^q$ is called \textit{Remotely Almost Periodic} (RAP, for short) if for every $\epsilon>0$ we have that the set
$$
T(f,\epsilon):=\Biggl\{ \tau \in {\mathbb R}: \limsup_{|t|\rightarrow +\infty}\| f(t+\tau)-f(t)\|<\epsilon\Biggr\}
$$
is relatively dense in $ {\mathbb R}.$  We denote by $RAP (\mathbb R\ : \ \mathbb C^q)$ the set of all such functions.
\end{defn}

Any number $\tau \in T(f,\epsilon)$ is called a remote $\epsilon$-translation number of $f$. We know that $RAP({\mathbb R } : {\mathbb C }^{q})$ is a closed subspace of 
$BUC({\mathbb R } : {\mathbb C }^{q})$ and therefore the Banach space itself. 		
If the functions $F_{1}, \cdot \cdot \cdot, F_{k}$ are remotely almost periodic ($k\in {\mathbb N}$), then for each $\epsilon>0$ the set of their common
remote $\epsilon$-translation number is relatively dense in ${\mathbb R};$ see e.g., \cite[Proposition 2.3]{zhang_parabolic}.

Our first main result concerns the linear non-homogeneous equation \eqref{DEPCA_lineal+f}.

\begin{thm}\label{existencia_solucion_depca_rap}
Suppose that $A$ and $B$ are remotely almost periodic matrix valued functions, and $f \in \mathbb{Z}RAP(\mathbb{R}, \mathbb{R}^q)$.  
Assume that the discrete equation associated with \eqref{DEPCA_lineal} admits an exponential dichotomy on $\mathbb{Z}$ and that its Green's kernel is bi-remotely almost periodic and summable.  
Then, equation \eqref{DEPCA_lineal+f} has a unique remotely almost periodic solution.
\end{thm}

Here $\mathbb{Z}RAP(\mathbb{R}, \mathbb{R}^q)$ denotes the space of $\mathbb{Z}$-remotely almost periodic functions, which captures the appropriate regularity for functions involving the piecewise constant argument $[t]$; see Definition \ref{def_ZRAP}.

We then consider perturbed systems of the form
\begin{eqnarray}
y'(t) &=& A(t)y(t) + B(t)y([t]) + f(t) + g_{\nu}(t, y(t), y([t])), \label{intro_perturbed}
\end{eqnarray}
where $g_{\nu}$ is a perturbation depending on a small parameter $\nu \geq 0$ with $g_0 \equiv 0$.

\begin{thm}\label{teo_perturbada_depca}
Let $\xi$ be the unique remotely almost periodic solution of \eqref{DEPCA_lineal+f}.  
If conditions \ref{H1} and \ref{H2} are satisfied,  
then there exists $r > 0$ and $\nu_0 > 0$ sufficiently small such that \eqref{intro_perturbed}  
has a unique remotely almost periodic solution $\psi_\nu$ in an $r$-neighborhood of $\xi$  
for each $\nu \in [0, \nu_0)$.  
Moreover, if $g_\nu(t, x, y)$ is uniformly continuous in $(t, x, y) \in \mathbb{R} \times B_r(0) \times B_r(0)$,  
 $\psi_\nu$ depends continuously on $\nu$ and  
\[
\lim_{\nu \to 0} \psi_\nu(t) = \xi(t), \quad \mbox{for all}~{} t\in\mathbb{R}.
\]
\end{thm}

We also extend these results to fully nonlinear equations of the form $y'(t) = \tilde{f}(t, y(t), y([t])) + g_{\nu}(t, y(t), y([t]))$ (see Theorem \ref{perturbed_f_nolineal}).

\subsection{Organization and main ideas of this work} 
In Section \ref{sec2}, we recall some necessary facts about remotely almost periodic functions, and introduce the notion of $\mathbb{Z}$-remotely almost periodic functions (see Definition \ref{def_ZRAP}), which is essential for handling the discontinuities inherent in DEPCAs.

Section \ref{sec3} and Section \ref{sec4} are devoted to the general theory where we derive the variation of parameters formula \eqref{DEPCA_var_para_local} and establish the connection between solutions of \eqref{DEPCA_lineal+f} and the associated difference equation \eqref{lineal+h_discreta}. We also introduce the notion of ($\alpha$, K, P)-exponential dichotomy for the discrete system and prove Theorem \ref{difference_sol_bounded}, which guarantees the existence and uniqueness of remotely almost periodic solutions for non-homogeneous difference equations under appropriate conditions on the Green's kernel.

In Section \ref{sec5}, we present one of our main results, Theorem \ref{existencia_solucion_depca_rap}, where we establish that if the matrices $A(t)$ and $B(t)$ are remotely almost periodic, $f(t)$ is $\mathbb{Z}$-remotely almost periodic, and the associated discrete equation admits an exponential dichotomy with a bi-remotely almost periodic and summable Green's kernel, then equation \eqref{DEPCA_lineal+f} has a unique remotely almost periodic solution. 

Section \ref{sec6} extends these results to perturbed systems of the form \eqref{depca_lineal+f_perturbada} and \eqref{depca_nolineal_perturbada}; the main result here is Theorem \ref{teo_perturbada_depca}, which shows that for sufficiently small $\nu$, the perturbed system has a unique remotely almost periodic solution in an $r$-neighborhood of the unperturbed solution, with continuity in the parameter $\nu$
. We also treat the case of nonlinear unperturbed equations in Theorem \ref{perturbed_f_nolineal}. Finally, in Section \ref{sec7}, as an application, we study the Lasota-Wazewska model with piecewise constant argument \eqref{application_depca}, which arises in the study of red blood cell dynamics. The main result of this section is Theorem \ref{gamma_small_unique_sol}, where we prove that under conditions:

\begin{enumerate}[label=\textbf{(LW.\arabic*)}]
\item\label{ILW1} The average of $\delta$ satisfies $M(\delta) > \delta_{-} > 0$.
\item\label{ILW2} The functions $\delta(\cdot)$, $p(\cdot)$ are positive remotely almost-periodic functions,  
and $f(\cdot)$ is a positive function that satisfies a $\gamma$-Lipschitz condition.
\end{enumerate}
Then for a sufficiently small $\gamma$, the model admits a unique positive remotely almost periodic solution.

We use the standard notation throughout the paper. By $BUC({\mathbb R } : {\mathbb C }^{q})$ we denote the Banach space of bounded and uniformly continuous functions $f : {\mathbb R } \rightarrow {\mathbb C }^{q},$ 
equipped with the sup-norm $\| \cdot \|_{\infty};$ let $\| \cdot \|$ be a fixed norm in ${\mathbb C}^{q}.$ We set ${\mathbb N}_{n}:=\{1,\cdot \cdot \cdot, n\}.$

\section{Preliminars on Remotely Almost Periodic Sequences}\label{sec2}

Analogously to remotely almost periodic functions, we will make a brief analysis of remotely almost periodic sequences. Throughout this paper, we use the variable $t$ for continuous time ($t \in \mathbb{R}$) and $n$ for discrete time ($n \in \mathbb{Z}$). When a function $f : \mathbb{R} \to \mathbb{R}^m$ appears as $f(n)$, it is understood that we consider the sequence $\{f(n)\}_{n \in \mathbb{Z}}$ obtained by restricting $f$ to the integers.

\begin{defn}
Let $x:\mathbb{Z}\rightarrow\mathbb{R}^{m}$ be a bounded sequence. We say that $x$ is:
\begin{enumerate}[label=(\roman*)]
\item \textbf{Slowly Oscillating} if for each $a\in\mathbb{Z}$,  
\[
\lim_{\left|n\right|\rightarrow\infty}\left\Vert x\left(n+a\right)-x\left(n\right)\right\Vert=0.
\]  
We denote the set of such sequences by $\mathcal{SO}\left(\mathbb{Z},\mathbb{R}^{m}\right)$.

\item \textbf{Almost Periodic} if for every $\epsilon > 0$, the set  
\[
T\left(x,\epsilon\right)=\left\{ \tau\in\mathbb{Z} \,\middle\vert\, \left\Vert x\left(n+\tau\right)-x\left(n\right)\right\Vert <\epsilon,\ \forall n\in \mathbb{Z} \right\}
\]  
is relatively dense in $\mathbb{Z}$. The number $\tau\in T\left(x,\epsilon\right)$ is called an \textbf{$\epsilon$-almost period}. We denote the set of such sequences by $AP\left(\mathbb{Z},\mathbb{R}^{m}\right)$.

\item \textbf{Remotely Almost Periodic} if for every $\epsilon > 0$, the set  
\[
T\left(x,\epsilon\right)=\left\{ \tau\in\mathbb{Z} \,\middle|\, \limsup_{\left|n\right|\rightarrow\infty}\left\Vert x\left(n+\tau\right)-x\left(n\right)\right\Vert<\epsilon \right\}
\]  
is relatively dense in $\mathbb{Z}$. The number $\tau\in T\left(x,\epsilon\right)$ is called a \textbf{remote $\epsilon$-translation number}. We denote the set of such sequences by $RAP\left(\mathbb{Z},\mathbb{R}^{m}\right)$.
\end{enumerate}
\end{defn}

An example of a sequence $x\in RAP(\mathbb{Z},\mathbb{R})$ is  
\begin{eqnarray*}
x\left(n\right)=\cos\left(n\right)+\cos\left(\sqrt{2}n\right)+\sin(n+\sqrt{\vert n\vert })+\frac{3n^{2}}{n^{2}+1}.
\end{eqnarray*}   

We now state the following theorem to relate a remotely almost periodic sequence to a remotely almost periodic function.

\begin{thm}\label{RAP_cont_RAP_disc}
A necessary and sufficient condition for a bounded sequence $\{u_{n}\}_{n\in\mathbb{Z}}$ to be remotely almost periodic is the existence of a function $f\in RAP\left(\mathbb{R},\mathbb{R}^{m}\right)$ such that $f(n)=u_{n},\, n\in\mathbb{Z}$.
\end{thm}
\begin{proof}
Suppose that $f\in RAP(\mathbb{R},\mathbb{R}^{m})$.  
 
As $RAP(\mathbb{R},\mathbb{R}^{m})$ is a closed subalgebra of $BUC(\mathbb{R},\mathbb{R}^n)$ generated by $AP(\mathbb{R},\mathbb{R}^{m})$ and $SO(\mathbb{R},\mathbb{R}^{m})$ we know that given $\epsilon>0$, there exist $g_{1},\, g_{2}\in AP\left(\mathbb{R},\mathbb{R}^{m}\right)$  
and $\varphi_{1},\,\varphi_{2}\in \mathcal{SO}\left(\mathbb{R},\mathbb{R}^{m}\right)$  
such that  
\[
\left\Vert f-\left[g_{1}+\varphi_{1}+g_{2}\varphi_{2}\right]\right\Vert_{\infty} <\epsilon/4.
\]  
By Theorem 1.7.3 in \cite{zhang_APTE}, $g_1(n), g_2(n)\in AP(\mathbb{Z},\mathbb{R}^{m})$.

Let $\epsilon'=\min\left\{ \epsilon/4,\epsilon/\left(4\left\Vert \varphi_{2}\right\Vert_{\infty}\right)\right\}$  
and let $\tau$ be an $\epsilon'$-translation number for $g_{1}(n)$ and $g_{2}(n)$.  
We will show that if $\tau$ is an $\epsilon'/2$-remote translation number  
for $g_{1}(n)+\varphi_{1}(n)+g_{2}(n)\varphi_{2}(n)$,  
then it is a remote $\epsilon$-translation number for $f(n)=u_{n}$. Since  
\[
\Vert f(n+\tau)-f(n)\Vert \leq \epsilon /2 +(1+\Vert \varphi_{2}\Vert_{\infty} )\epsilon' +\Vert \Delta_{\tau}\varphi_{1}(n)\Vert +\Vert g_{2}\Vert_{\infty} \Vert \Delta_{\tau}\varphi_{2}(n)\Vert,
\]  
it follows that  
\[
\limsup_{\left|n\right|\rightarrow\infty}\Vert f(n+\tau)-f(n)\Vert<\epsilon.
\]  
This shows that $T\left(f(n),\epsilon\right)$ is relatively dense in $\mathbb{Z}$ and $f(n)$ is remotely almost periodic.

Now for the converse, assume that $u_n$ is a remotely almost periodic sequence.  
Define the function $f:\mathbb{R} \to \mathbb{R}^m$ by  
\[
f(t) = u_n + (t - n)(u_{n+1} - u_n),\quad\text{for } n \leq t < n+1,\; n\in\mathbb{Z}.
\]  
It is clear that $f(n) = u_n$. We will show that $f \in RAP(\mathbb{R},\mathbb{R}^m)$;  
we only need to show that $T(u_n, \epsilon/3) \subset T(f, \epsilon)$.  
Let $\tau \in T(u_n, \epsilon/3)$. If $n \leq t < n+1$, then $n+\tau \leq t+\tau < n+\tau+1$.  
Since  
\[
f(t+\tau)-f(t)= f(n+\tau) - f(n) + (t - n)\left\{ [f(n+\tau+1) - f(n+1)] - [f(n+\tau) - f(n)] \right\},
\]
and $0 \leq t - n < 1$, we obtain  
\[
\limsup_{\left|t\right|\rightarrow\infty} \Vert f(t+\tau) - f(t) \Vert < \epsilon.
\]  
Thus, $\tau \in T(f,\epsilon)$, which completes the proof.
\end{proof}

Given the continuous function $f$, it is evident that the function $f([\cdot])$ is discontinuous (unless $f$ is a constant function).  
From this, we observe that for the function $\varphi \in RAP(\mathbb{R},\mathbb{R})$, the function $\varphi([\cdot])$ is not remotely almost periodic, since it does not satisfy uniform continuity.  
As naturally expected, the discontinuities occur at the integers; furthermore, it is clear that the lateral limits exist for every $n \in \mathbb{Z}$.  
In \cite{zhang_depca}, there is no precise definition of what is meant by remotely almost periodicity of the possibly discontinuous function $f(\cdot, x(\cdot), x([\cdot]))$, which leads to some inaccuracies.  

Analogous to the concept introduced in the context of almost periodic and almost automorphic functions by E. Ait Dads and L. Lachimi \cite{lachimi}, and A. Ch\'{a}vez, S. Castillo, and M. Pinto \cite{alan}, respectively,  
we will later see that it is sufficient to consider \textbf{$\mathbb{Z}$ Remotely Almost Periodic} functions to study solutions of DEPCAs.

We define a subset of the bounded functions as  

\begin{eqnarray*}
BR(\mathbb{R},\mathbb{R}^{n}) :=
\left\{
 f:\mathbb{R} \rightarrow \mathbb{R}^{n}  \left| 
 \begin{aligned}
\quad f \text{ is bounded, continuous on } \mathbb{R}\setminus\mathbb{Z}, \\
\quad  \text{and has finite lateral limits at each point in } \mathbb{Z}
\end{aligned}
\right.
\right\}
\end{eqnarray*}

\begin{defn}\label{def_ZAP}
A function $f \in BR(\mathbb{R}, \mathbb{R}^n)$ is said to be \textbf{$\mathbb{Z}$ almost periodic} if there exists a set $T(f, \epsilon) \cap \mathbb{Z}$ that is relatively dense in $\mathbb{R}$,  
and for every $\tau \in T(f, \epsilon) \cap \mathbb{Z}$, the following holds:  
\begin{eqnarray*}
\Vert f(\cdot + \tau) - f(\cdot) \Vert_{\infty} < \epsilon.
\end{eqnarray*}  
We denote the set of such functions by $\mathbb{Z}AP(\mathbb{R}, \mathbb{R}^n)$.
\end{defn}

\begin{defn}\label{def_ZRAP}
A function $f \in BR(\mathbb{R}, \mathbb{R}^n)$ is said to be \textbf{$\mathbb{Z}$ remotely almost periodic} if there exists a relatively dense set $T_1(f, \epsilon) \cap \mathbb{Z}$ such that,  
if $\tau \in T_1(f, \epsilon) \cap \mathbb{Z}$, then  
\begin{eqnarray*}
\limsup_{\vert t \vert \rightarrow \infty} \Vert f(t+\tau)-f(t) \Vert < \epsilon,\quad t \in \mathbb{R}.
\end{eqnarray*}  
We denote the set of such functions by $\mathbb{Z}RAP(\mathbb{R}, \mathbb{R}^n)$.
\end{defn}

\begin{prop}
The set of functions $\mathbb{Z}AP(\mathbb{R},\mathbb{R}^n)$ is a subset of $\mathbb{Z}RAP(\mathbb{R},\mathbb{R}^n)$.
\end{prop}
\begin{proof}
If $f \in \mathbb{Z}\mathrm{AP}\left(\mathbb{R},\mathbb{R}^{n}\right)$, then the set  
\[
T_{1}\left(f,\epsilon\right)\cap \mathbb{Z} = \left\{ a\in\mathbb{R} \,\middle|\, \left\Vert f\left(\cdot+a\right)-f\left(\cdot\right)\right\Vert_{\infty} <\epsilon \right\} \cap \mathbb{Z}
\]  
is relatively dense in $\mathbb{R}$.  
Since $T_{1}\left(f,\epsilon\right) \subset T\left(f,\epsilon\right)$, it follows that  
$f \in \mathbb{Z}\mathrm{RAP}\left(\mathbb{R},\mathbb{R}^{n}\right)$.
\end{proof}

\begin{rem}
Every function in $\bb{Z}RAP(\bb{R},\bb{R}^n)$ is locally integrable.
\end{rem}

Recall the following Lemma which it can be found in \cite{corduneanu_waves}.
\begin{lem}\label{AP_TcapZ_novacio}
Let $f \in AP(\mathbb{R})$. Then $T(f,\epsilon) \cap \mathbb{Z} \neq \emptyset$ and it is relatively dense.
\end{lem}  

Using this we can now prove the next Lemma.
\begin{lem}\label{RAP_TcapZ_novacio}
Let $f \in RAP(\mathbb{R})$. Then $T(f,\epsilon) \cap \mathbb{Z} \neq \emptyset$, and it is relatively dense.
\end{lem}
\begin{proof}
As $RAP(\mathbb{R},\mathbb{R}^{m})$ is generated by $AP(\mathbb{R},\mathbb{R}^{m})$ and $SO(\mathbb{R},\mathbb{R}^{m})$, given $\epsilon > 0$, there exists $\psi\in RAP(\mathbb{R},\mathbb{R}^n)$ given by $\psi(t) = g_1(t) + \varphi_1(t) + g_2(t)\varphi_2(t)$,  
where $g_1, g_2 \in AP(\mathbb{R}, \mathbb{R}^n)$ and $\varphi_1, \varphi_2 \in OS(\mathbb{R}, \mathbb{R}^n)$, such that  
\begin{eqnarray*}
\left\Vert f(t) - \psi(t) \right\Vert \leq \epsilon / 4, \quad \forall t\in \mathbb{R}.
\end{eqnarray*}  
Consider $\tau \in T(g_1, \epsilon/3) \cap T(g_2, \epsilon/3) \cap \mathbb{Z}$, with  
\[
\epsilon' = \min\left\{ \frac{\epsilon}{4}, \frac{\epsilon}{4\Vert \varphi_2\Vert_{\infty}} \right\}.
\]  
We know this set is nonempty by Lemma \ref{AP_TcapZ_novacio}.  
Therefore, we have
\begin{eqnarray*}
\Vert f(t+\tau)-f(t) \Vert 
&\leq & \Vert f(t+\tau)-\psi(t+\tau)\Vert + \Vert \psi(t+\tau)-\psi(t) \Vert + \Vert \psi(t)-f(t)\Vert \\
&\leq & \epsilon/2 + \Vert \varphi_{1}(t+\tau)- \varphi_{1}(t)\Vert \\
& & \quad \,\,\,\, +\Vert g_{2}\Vert_{\infty} \Vert \varphi_{2}(t+\tau)- \varphi_{2}(t)\Vert + \Vert \varphi_{2}\Vert_{\infty} \Vert g_{2}(t+\tau)-g_{2}(t)\Vert
\end{eqnarray*}
Finally,
\begin{eqnarray*}
\limsup_{\vert t\vert \rightarrow \infty} \Vert f(t+\tau)-f(t) \Vert 
&\leq & \epsilon
\end{eqnarray*}
Thus, we have that $\tau \in T(f, \epsilon) \cap \mathbb{Z} \neq \emptyset$.
\end{proof}

\begin{prop}
If $f$ is a remotely almost periodic function, then $f([\cdot]) \in \mathbb{Z}RAP(\mathbb{R}, \mathbb{R}^n)$.
\end{prop}
\begin{proof}
Let us consider $\tau \in T(f, \epsilon) \cap \mathbb{Z}$; by Lemma \ref{RAP_TcapZ_novacio}, we know that this set is nonempty.  
Since $[t+\tau] = [t] + \tau$.  
Finally,
\begin{eqnarray*}
\limsup_{\vert t\vert \rightarrow \infty} \Vert f([t]+\tau)-f([t]) \Vert 
&\leq & \epsilon.
\end{eqnarray*}
Therefore, $f([\cdot]) \in \mathbb{Z}RAP(\mathbb{R}, \mathbb{R}^n)$.
\end{proof}

\begin{lem}
Let $f \in RAP(\mathbb{R} \times \mathbb{R}^n, \mathbb{R}^n)$ and assume it is Lipschitz continuous in the second variable with constant $K$.  
If $\varphi \in \mathbb{Z}RAP(\mathbb{R}, \mathbb{R}^n)$, then  
\[
f(\cdot, \varphi(\cdot)) \in \mathbb{Z}RAP(\mathbb{R}, \mathbb{R}^n).
\]
\end{lem}
\begin{proof}
The proof is straightforward. If we consider $\tau \in T(f, \epsilon') \cap T(\varphi, \epsilon') \cap \mathbb{Z}$, where $\epsilon' = \frac{\epsilon}{1+K}$, it follows that
\begin{eqnarray*}
\Vert f(t+\tau,\varphi(t+\tau))-f(t,\varphi(t)) \Vert 
&\leq & \Vert f(t+\tau,\varphi(t+\tau))-f(t+\tau,\varphi(t)) \Vert \\
& & +\Vert f(t+\tau,\varphi(t))-f(t,\varphi(t))\Vert\\
&\leq& K\Vert \varphi(t+\tau)-\varphi(t) \Vert + \Vert f(t+\tau,\varphi(t))-f(t,\varphi(t))\Vert
\end{eqnarray*}
Thus,
\begin{eqnarray*}
\limsup_{\vert t \vert \rightarrow\infty} \Vert f(t+\tau,\varphi(t+\tau))-f(t,\varphi(t)) \Vert 
&\leq & \limsup_{\vert t \vert \rightarrow\infty} K\Vert \varphi(t+\tau)-\varphi(t) \Vert \\
& &+ \limsup_{\vert t \vert \rightarrow\infty} \Vert f(t+\tau,\varphi(t))-f(t,\varphi(t))\Vert\\
&\leq & \epsilon.
\end{eqnarray*}
Which it ends the proof.
\end{proof}

\begin{lem}
Let $f:\mathbb{R}\times \mathbb{R}^n\times \mathbb{R}^n \rightarrow \mathbb{R}^n$ be remotely almost periodic and Lipschitz in the second and third variables, and let $\varphi \in RAP(\mathbb{R},\mathbb{R}^n)$. Then  
\[
f(\cdot,\varphi(\cdot),\varphi([\cdot])) \in \mathbb{Z}RAP(\mathbb{R},\mathbb{R}^n).
\]
\end{lem}
\begin{proof}
The proof is straightforward.
\end{proof}

\begin{lem}\label{ZRAP_unif_conti_RAP}
Let $f$ be a $\mathbb{Z}$ remotely almost periodic function. If $f$ is continuous on $\mathbb{R}$, then $f$ is remotely almost periodic.
\end{lem}
\begin{proof}
It is immediate, since the function $f$ is continuous. And because it is $\mathbb{Z}$ remotely almost periodic, there exists a relatively dense set $T(f,\epsilon)$ such that if $\tau \in T(f,\epsilon)$, then $\forall t \in \mathbb{R}$  we have
\begin{eqnarray*}
\limsup_{\vert t\vert \rightarrow \infty} \Vert f(t+\tau)-f(t) \Vert 
&\leq & \epsilon.
\end{eqnarray*}
\end{proof}

\section{Differential Equations with Piecewise Constant Argument}\label{sec3}

Let $\phi:\mathbb{R}\rightarrow M_{q\times q}(\mathbb{R})$ be the fundamental matrix of \eqref{DEPCA_A}, and let the transition matrix  $\Phi:\mathbb{R}\times \mathbb{R}\rightarrow M_{q\times q}(\mathbb{R})$ given by $\Phi(t,s) = \phi(t)\phi^{-1}(s)$ for all $t,s\in\mathbb{R}$.  Using the variation of parameters formula on $[n, n+1)$, the solution of \eqref{DEPCA_lineal+f} satisfies
\begin{eqnarray}
x(t)&=& \left[\Phi(t,n)+\int_{n}^{t}\Phi(t,u)B(u)du\right]x(n)+\int_{n}^{t}\Phi(t,u)f(u)du \label{DEPCA_var_para_local} \\
& & \ \ \ \ \ \ \ \ \ \ \ \ \ \ \ \ \ \ t\in \mathbb{R},\ n=[t], n\leq t< n+1. \nonumber
\end{eqnarray}
We also define  
\begin{eqnarray}
J(t,s) &=& I + \int_{s}^{t} \Phi(s,u) B(u)\,du \label{J_depca} \\
Z(t,s) &=& \Phi(t,s) J(t,s) \label{Z_depca}
\end{eqnarray}  
for all $t,s\in\mathbb{R}$. We require $J$ to be invertible; see \cite{cauchy,kuo} and \cite{robledo}.  

By Definition \ref{def_sol_DEPCA}, we have that \eqref{DEPCA_var_para_local} is continuous at $t = n+1$,  
so if we consider $t \to (n+1)^-$, we obtain

\begin{eqnarray}
x(n+1)=C(n)x(n)+h(n),\quad \forall n\in \mathbb{Z},\label{discreta_DEPCA}
\end{eqnarray}
the discrete equation associated with \eqref{DEPCA_lineal+f}, where
\begin{eqnarray}
C(n)&=& \Phi(n+1,n)+\int_{n}^{n+1}\Phi(n+1,u)B(u)du=Z(n+1,n) \label{C(n)} \\
h(n)&=& \int_{n}^{n+1}\Phi(n+1,u)f(u)du. \label{h(n)}
\end{eqnarray}
By the continuity of the solution of a DEPCA, we have obtained a difference equation; that is, we have a recursive formula for each $n \in \mathbb{Z}$,  
which allows us to pass from one interval to the next.  
Finally, we can conclude that, to solve a DEPCA, one must solve a difference equation.  

Now, applying the variation of parameters to the equation
\eqref{discreta_DEPCA}, we get
\begin{eqnarray}
x(k) & = & \left(\prod_{i=s}^{k-1}C(i)\right)x(s)+\left(\prod_{m=s}^{k-1}C(m)\right)\left(\sum_{i=s}^{k-1}\left(\prod_{j=s}^{i}C(j)\right)^{-1}h(i)\right)\nonumber \\
x(k) & = & \left(\prod_{i=s}^{k-1}Z(i+1,i)\right)x(s)+\left(\sum_{i=s}^{k-1}\prod_{j=i+1}^{k-1}Z(j+1,j)h(i)\right)\label{sol_discreta}
\end{eqnarray}
with $s\in\mathbb{Z}$ fixed.

Now, if we substitute \eqref{sol_discreta} into \eqref{DEPCA_var_para_local}, we obtain

\begin{eqnarray*}
x(t) & = & Z(t,n)\left(\left(\prod_{i=s}^{n-1}Z(i+1,i)\right)x(s)
+\left(\sum_{i=s}^{n-1}\prod_{j=i+1}^{n-1}Z(j+1,j)h(i)\right)\right)\\
& &+\int_{n}^{t}\Phi(t,u)f(u)du\\
 & = & Z(t,n)\left(\prod_{i=s}^{n-1}Z(i+1,i)\right)x(s)+\int_{n}^{t}\Phi(t,u)f(u)du \\
 &  & +Z(t,n)\left(\sum_{i=s}^{n-1}\prod_{j=i+1}^{n-1}Z(j+1,j)\int_{i}^{i+1}\Phi(i+1,u)f(u)du\right)\\
 & = & Z(t,n)\left(\prod_{i=s}^{n-1}Z(i+1,1)\right)x(s)   +\int_{n}^{t}\Phi(t,u)f(u)du\\
 &  & +\sum_{i=s}^{n-1}\int_{i}^{i+1}\left(\prod_{j=i+1}^{n-1}Z(t,n)Z(j+1,j)\Phi(i+1,u)\right)f(u)du
\end{eqnarray*}
Thus,
\begin{eqnarray}
x(t) & = & Z(t,n)\left(\prod_{i=s}^{n-1}Z(i+1,i)\right)x(s)+\int_{n}^{t}\Phi(t,u)f(u)du \nonumber \\
 &  & +\sum_{i=s}^{n-1}\int_{i}^{i+1}\left(\prod_{j=i+1}^{n-1}Z(t,n)Z(j+1,j)\right)\Phi(i+1,u)f(u)du
 . \label{sol_s_in_Z}
\end{eqnarray}
Now, if $s\notin \bb{Z}$ we know $[s]\leq s\leq [s]+1$.  From \eqref{DEPCA_var_para_local} we get
\begin{eqnarray*}
x(t)&=& \left[\Phi(t,s)+\int_{s}^{t}\Phi(t,u)B(u)du\right]x(s)+\int_{s}^{t}\Phi(t,u)f(u)du\\
& & \ \ \ \ \ \ \ \ \ \ \ \ \ \ \ \ \ \ t\in \mathbb{R},\  [s]\leq t< [s]+1. \nonumber
\end{eqnarray*}
Therefore, if $t\rightarrow ([s]+1)$
\begin{eqnarray}
x([s]+1)&=& Z([s]+1,s)x(s)+\int_{s}^{[s]+1}\Phi([s]+1,u)f(u)du \label{s_no_Z}
\end{eqnarray}
Then, by substituting \eqref{s_no_Z} into \eqref{sol_s_in_Z}, we obtain 
\begin{eqnarray*}
x(t) & = & Z(t,n)\left(\prod_{i=[s]+1}^{n-1}Z(i+1,i)\right)\left( Z([s]+1,s)x(s)+\int_{s}^{[s]+1}\Phi([s]+1,u)f(u)du\right)\\
 &  & +\sum_{i=[s]+1}^{n-1}\int_{i}^{i+1}\left(\prod_{j=i+1}^{n-1}Z(t,n)Z(j+1,j)\right)\Phi(i+1,u)f(u)du
 +\int_{n}^{t}\Phi(t,u)f(u)du\\
 &=&Z(t,n)\left(\prod_{i=[s]+1}^{n-1}Z(i+1,i)\right)Z([s]+1,s)x(s)\\
 &&+\int_{s}^{[s]+1}Z(t,n)\left(\prod_{i=[s]+1}^{n-1}Z(i+1,i)\right)\Phi([s]+1,u)f(u)du\\
 &&+\sum_{i=[s]+1}^{n-1}\int_{i}^{i+1}\left(\prod_{j=i+1}^{n-1}Z(t,n)Z(j+1,j)\right)\Phi(i+1,u)f(u)du
 +\int_{n}^{t}\Phi(t,u)f(u)du
\end{eqnarray*}
Then, we consider
\begin{eqnarray}
R(t,s)&=&\begin{cases}
Z(t,n)\left(\displaystyle{\prod_{i=s}^{n-1}}Z(i+1,i)\right) & \mbox{if }s\in\mathbb{Z}\\
Z(t,n)\left(\displaystyle{\prod_{i=[s]+1}^{n-1}}Z(i+1,i)\right)Z([s]+1,s) & \mbox{if }s\notin\mathbb{Z}.
\end{cases}\label{R_depca}
\end{eqnarray}
Moreover, if $s\in\bb{Z}$
\begin{eqnarray*}
\overline{G}(t,u)=\begin{cases}
\Phi(t,u) & \mbox{if }[t]<u\leq t\\
Z(t,s)\displaystyle{\prod_{j=i+1}^{n-1}}Z(j+1,j)\Phi(i,u) & \mbox{if }i<u<i+1\leq[t]<t
\end{cases}, 
\end{eqnarray*}
and if $s\notin\bb{Z}$
\begin{eqnarray*}
\tilde{G}(t,u)&=&\begin{cases}
\Phi(t,u) & \mbox{if }[t]<u\leq t\\
Z(t,[t])\displaystyle{\prod_{j=i+1}^{n-1}}Z(j+1,j)\Phi(i,u) & \mbox{if }i<u<i+1\leq[t]<t\\
Z(t,n)\left(\displaystyle{\prod_{i=[s]+1}^{n-1}}Z(i+1,i)\right)\Phi([s]+1,u) & \mbox{if }s\leq u\leq[s]+1
\end{cases}
\end{eqnarray*}
Thus,
\begin{eqnarray}
G(t,u)&=&\begin{cases}
\overline{G}(t,u) & \mbox{if }s\in\mathbb{Z}\\
\widetilde{G}(t,u) & \mbox{if }s\notin\mathbb{Z}.
\end{cases}\label{varpar_G}
\end{eqnarray}
Finally, we obtain a variation of parameters formula for the differential equation with piecewise constant argument, given by
\begin{eqnarray*}
x(t)=R(t,s)x(s)+\int_{s}^{t} G(t,\tau)f(\tau)d\tau.
\end{eqnarray*}

In the following example, we will see how to use the variation of parameters.

\begin{exam}
Let us consider the scalar linear equation with constant coefficients  
\begin{eqnarray*}
x'(t) = a x(t) + b x([t]) + c, 
\end{eqnarray*}  
where $a, b, c \in \mathbb{R}$ with $a \neq 0$, such that $x(0) = x_0$.  

We have that $\Phi(t,s) = e^{a(t - s)}$. Recalling \eqref{J_depca} and \eqref{Z_depca}, we have that

\begin{eqnarray*}
J(t,s)&=& 1+\frac{b}{a}(1-e^{-a(t-s)}) \\
Z(t,s)&=& e^{a(t-s)}+\frac{b}{a}(e^{a(t-s)}-1).
\end{eqnarray*}
Thus, by the operators defined in \eqref{R_depca} and \eqref{varpar_G}, we obtain

\begin{eqnarray*}
R(t,0)=(e^{a(t-n)}+\frac{b}{a}(e^{a(t-n)}+1)) \left(e^{a}+\frac{b}{a}(e^{a}-1)\right)^{n-1},\ t\in\bb{R} 
\end{eqnarray*}
and,
\[
G(t,u)=\begin{cases}
e^{a(t-u)} & \mbox{if }[t]<u\leq t\\
\left(e^{at}+\frac{b}{a}(e^{at}-1)\right) \displaystyle{\prod_{j=i+1}^{n-1}} (e^{a}+\frac{b}{a}(e^{a}-1))e^{a(i-u)} & \mbox{if }i<u<i+1\leq[t]<t
\end{cases} \label{G_depca_example}
\]
Therefore the solution is given by
\begin{eqnarray*}
x(t)=R(t,0)x_0+c\int_{0}^{t} G(t,\tau)d\tau.
\end{eqnarray*}
or represented by
\begin{eqnarray*}
x(t) & = & (e^{a(t-[t])}+\frac{b}{a}(e^{a(t-[t])}-1)) \left(e^{a}+\frac{b}{a}(e^{a}-1)\right)^{[t]} x_0\\
 &  & +\sum_{i=0}^{[t]-1}c \int_{i}^{i+1} \left(e^{at}+\frac{b}{a}(e^{at}-1)\right) \prod_{j=i+1}^{[t]-1} (e^{a}+\frac{b}{a}(e^{a}-1))e^{a(i-u)} du\\
 &&+c\int_{[t]}^{t}e^{a(t-u)}du.
\end{eqnarray*}
\end{exam}

From the above, we will first study remotely almost-periodic solutions in difference equations.

\section{Remotely Almost-Periodic Solutions in Difference Equations}\label{sec4}

Difference equations arise naturally in discrete mathematics and play a key role in various areas,  
such as probability, computer science, numerical methods for differential equations,  
and control theory (see \cite{agarwal,lakshmikantham,papaschinopoulos_2}).  
The properties of these equations and their applications to discrete dynamical systems and differential equations  
with piecewise constant arguments have been studied recently by several authors.  

In this subsection, we study the equations  
\begin{eqnarray}
x(n+1) &=& C(n)x(n), \label{lineal_discreta} \\
y(n+1) &=& C(n)y(n) + h(n), \label{lineal+h_discreta}
\end{eqnarray}  
where both $C:\mathbb{Z} \rightarrow \mathbb{R}^{q\times q}$, an invertible square matrix,  
and $h:\mathbb{Z} \rightarrow \mathbb{R}^q$ are remotely almost-periodic sequences.  
The system \eqref{lineal+h_discreta} is thoroughly studied in \cite{jialin}.

Our focus is on equation \eqref{lineal+h_discreta}, since discrete remotely almost-periodic solutions have not been studied in the literature.  
The almost-periodic version can be found in \cite{zhang_APTE}.  

In the previous chapter, exponential dichotomy played a key role in the study of remotely almost-periodic solutions in non-autonomous differential equations,  
mainly because it allows us to obtain a bounded solution dependent on the Green's kernel associated with the system.  
The following definition describes the discrete version of exponential dichotomy.  

\begin{defn}
We say that equation \eqref{lineal_discreta} has an $(\alpha, K, P)$-exponential dichotomy on $\mathbb{Z}$ if there exist positive constants $\alpha$, $K \geq 1$, and a projection $P$ $(P^2 = P)$, such that
\begin{eqnarray*}
\Vert \tilde{G}(n,m) \Vert \leq Ke^{-\alpha \vert n-m \vert},
\end{eqnarray*}
where $\tilde{G}(n,m)$ is the discrete Green's function given by
\begin{eqnarray*}
\tilde{G}(n,m)&=&\begin{cases}
\Phi\left(n\right)P\Phi^{-1}\left(m\right) & n\geq m\\
-\Phi\left(n\right)\left(I-P\right)\Phi^{-1}\left(m\right) & n< m
\end{cases}
\end{eqnarray*}
where $\Phi(n)$ is the fundamental matrix of system \eqref{lineal_discreta} such that $\Phi(0) = I$.
\end{defn}

In what follows, we will study the difference equation \eqref{lineal+h_discreta} and the existence of remotely almost-periodic solutions. We have the following theorem:

\begin{thm}\label{difference_sol_bounded}
Let $h \in RAP(\mathbb{Z}, \mathbb{R}^q)$.  
Suppose that the difference equation \eqref{lineal_discreta}  
has an $(\alpha, K, P)$-exponential dichotomy on $\mathbb{Z}$, and that the Green's function $\tilde{G}(\cdot,\cdot)$ is bi-remotely almost-periodic and summable, i.e., for every $\epsilon > 0$ there exists a relatively dense set $T(\tilde{G}, \epsilon) \subset \mathbb{Z}$ such that if $\tau \in T(\tilde{G}, \epsilon)$, then  
\begin{eqnarray}
\limsup_{|n| \rightarrow \infty} \sum_{k \in \mathbb{Z}} \left\Vert \tilde{G}(n+\tau, k+\tau) - \tilde{G}(n,k) \right\Vert \leq \epsilon. \label{bi_sumable}
\end{eqnarray}  
Then, the difference equation \eqref{lineal+h_discreta} has a unique remotely almost-periodic solution $y(n)$, given by  
\begin{eqnarray*}
y(n) = \sum_{k \in \mathbb{Z}} \tilde{G}(n,k)h(k), \quad n \in \mathbb{Z}.
\end{eqnarray*}  
Moreover, it is uniformly bounded by
\[
\left\Vert y\left(n\right)\right\Vert \leq K\left(1+e^{-\alpha}\right)\left(1-e^{-\alpha}\right)^{-1}\left\Vert h\right\Vert_{\infty},\, \, n\in \bb{Z}.
\]
\end{thm}
\begin{proof}
In \cite[Theorem 5.7]{zhang_APTE}, it is shown that the unique bounded solution of the discrete equation \eqref{lineal+h_discreta} is  
\[
y(n) = \sum_{k \in \mathbb{Z}} \tilde{G}(n,k)h(k).
\]  
Therefore, we will only prove that this solution is remotely almost-periodic.

Indeed, let us consider
\begin{eqnarray*}
\Vert y(n+\tau)-y(n)\Vert 
&=& \left\Vert \sum_{k\in \bb{Z}} \tilde{G}(n+\tau ,k+\tau)h(k+\tau)-\sum_{k\in \bb{Z}} \tilde{G}(n,k)h(k) \right\Vert\\
&\leq& \left\Vert \sum_{k\in \bb{Z}} (\tilde{G}(n+\tau ,k+\tau)- \tilde{G}(n,k))h(k+\tau) \right\Vert \\
& &+ \left\Vert \sum_{k\in \bb{Z}} \tilde{G}(n,k)(h(k+\tau)-h(k)) \right\Vert \\
&\leq&\sum_{k\in\mathbb{Z}}\left\Vert \tilde{G}(n+\tau,k+\tau)-\tilde{G}(n,k)\right\Vert \left\Vert h(k+\tau)\right\Vert\\
&&+K\sum_{k\in\mathbb{Z}}e^{-\alpha\left|n-k\right|}\left\Vert h(k+\tau)-h(k)\right\Vert.
\end{eqnarray*}
Taking advantage of the fact that the Green's function is bi-remotely almost-periodic and summable,  
that the function $h$ is remotely almost-periodic, and the exponential dichotomy, we obtain the result.
Let us define
\begin{eqnarray*}
S_1(n) &=& \sum_{k\in\bb{Z}} \Vert \tilde{G}(n+\tau ,k+\tau)-\tilde{G}(n,k)\Vert \Vert h(k+\tau)\Vert \label{S_1}\\
S_2(n) &=&  \sum_{k\in\bb{Z}} e^{-\alpha\vert n-k\vert} \Vert h(k+\tau)-h(k)\Vert.
\end{eqnarray*}
We know that $\Vert h \Vert_{\infty} < M$. Let us study only the case when $n \rightarrow \infty$. The case $n \rightarrow -\infty$ is analogous.

Given $\epsilon > 0$, let us consider
 $\epsilon'=\min\{\epsilon(2\sum_{k\in\bb{Z}}e^{-\alpha\vert n-k\vert}+4M)^{-1},\epsilon(2M)^{-1}\}$, there exists $n_1>0$ such that if $\tau\in T(h,\epsilon')$
\begin{eqnarray}
\left\Vert h(n+\tau)-h(n) \right\Vert < \epsilon',\ \mathrm{if}\ \vert n\vert \geq n_1 .\label{rap_h_n1}
\end{eqnarray}
Note that  
\begin{eqnarray*}
\sum_{-n_1}^{n_1} e^{\alpha k} < M_1,
\end{eqnarray*}  
since it is a finite sum.  

Moreover, there exists $n_2 > 0$ such that  
\begin{eqnarray}
e^{-\alpha n} M_1 < \epsilon',\quad \text{if } n \geq n_2. \label{exp_n2}
\end{eqnarray}  
Let $n_0 = \max\{n_1, n_2\}$. Then, if $n > n_0$, for
\begin{eqnarray*}
S_2(n) 
&= & \left(\sum_{-\infty}^{-n_1}+\sum_{-n_1}^{n_1}+\sum_{n_1}^{\infty}\right)  e^{-\alpha\vert n-k\vert} \left\Vert h(k+\tau)-h(k)\right\Vert \\
    &\leq & \sup_{k\in(-\infty, -n_1]} \left\Vert h(k+\tau)-h(k)\right\Vert \sum_{-\infty}^{-n_1}  e^{-\alpha\vert n-k\vert} \\
    & & + \sum_{-n_1}^{n_1} e^{-\alpha \vert n-k \vert} \left\Vert h(k+\tau)-h(k) \right\Vert\\
    & &+\sup_{k\in[n_1,\infty)}\left\Vert h(k+\tau)-h(k)\right\Vert \sum_{n_1}^{\infty}  e^{-\alpha\vert n-k\vert} \\
    &\leq & 2\epsilon' \sum_{-\infty}^{\infty}   e^{-\alpha\vert n-k\vert}+ 2M\sum_{-n_1}^{n_1} e^{-\alpha ( n-k)} \\
    &\leq & 2\epsilon'\sum_{-\infty}^{\infty}   e^{-\alpha\vert n-k\vert} + 2M e^{-\alpha n} \sum_{k=-n_1}^{n_1} e^{\alpha k } \\
        &\leq & \frac{\epsilon}{2}
\end{eqnarray*}
From \eqref{rap_h_n1} and \eqref{exp_n2}, we have that  
\[
S_2(n) < \frac{\epsilon}{2},\quad |n| > n_0.
\]  
We then conclude that

\begin{eqnarray*}
\limsup_{\vert n \vert\rightarrow \infty}S_2(n)<\frac{\epsilon}{2}.
\end{eqnarray*}
Then, if we consider $\tau \in T(h, \epsilon') \cap T(\tilde{G}, \epsilon') \cap \mathbb{Z}$, we have for \eqref{S_1}

\begin{eqnarray*}
S_1(n) &\leq & M \sum_{-\infty}^{\infty} \Vert \tilde{G}(n+\tau ,k+\tau)-\tilde{G}(n,k)\Vert.
\end{eqnarray*}

Using that the Green's kernel satisfies \eqref{bi_sumable}, we obtain

\begin{eqnarray*}
\limsup_{\vert n \vert \rightarrow \infty} S_1(n)
 &\leq & M\limsup_{\vert n \vert \rightarrow \infty} \sum_{-\infty}^{\infty} \Vert \tilde{G}(n+\tau ,k+\tau)-\tilde{G}(n,k)\Vert\\
&\leq &  \frac{\epsilon}{2}
 \end{eqnarray*}
Thus, finally we have
\begin{eqnarray*}
\limsup_{n\rightarrow \infty} \left\Vert y(n+\tau)-y(n)\right\Vert  & <&  \epsilon,
\end{eqnarray*}
That is, $y$ is a remotely almost-periodic function.  
This concludes the proof.
\end{proof}

\begin{rem}
Similar conditions to those presented in Section 1.3 can be given to ensure that \eqref{bi_sumable} is satisfied.
\end{rem}

From the previous result, the following corollary follows:

\begin{cor}
Let $h \in RAP(\mathbb{Z}, \mathbb{R}^q)$. Suppose that $C(n) \equiv C$ in \eqref{lineal+h_discreta} has eigenvalues outside the unit circle (that is, with modulus different from $1$).  
Then the difference equation \eqref{lineal+h_discreta} has a unique remotely almost-periodic solution $y(n)$.
\end{cor}

\begin{rem}
The Schur-Cohn criterion \cite{schur-cohn} provides necessary and sufficient conditions for the eigenvalues of a square matrix to lie outside the unit circle.
\end{rem}

Finally, let us consider the non-homogeneous linear system
\begin{eqnarray}
y(n+1)	=	C(n)y(n)+h(n)+f_{\nu}(n,y(n)), \quad n\in\mathbb{Z},\label{difference_perturbed}
\end{eqnarray}
 such that $f_{\nu}(n, z) = f(n, z, \nu)$ and moreover $\Vert f_{\nu} \Vert \rightarrow 0$ as $\nu \rightarrow 0$ uniformly in $(n, z) \in \mathbb{Z} \times B_r(0)$.  
Let us consider the following hypotheses:
\begin{enumerate}[label=\textbf{(D.\arabic*)}]
\item\label{D1} The difference equation \eqref{lineal_discreta}  
has an $(\alpha, K, P)$-exponential dichotomy on $\mathbb{Z}$, and the Green's function $\tilde{G}(\cdot,\cdot)$ is bi-remotely almost-periodic and summable, that is, \eqref{bi_sumable} holds.  

\item\label{D2} Let $f_\nu(n,z) = f(n,z,\nu)$ be a remotely almost-periodic sequence in $n$, for $(z,\nu) \in B_r(0) \times [0, \nu_0]$, and for each fixed real parameter $\nu$, it is uniformly bounded with respect to $z$.  
Moreover, it satisfies the Lipschitz condition  
\begin{eqnarray*}
\Vert f_{\nu}(n,x) - f_{\nu}(n,z) \Vert \leq M_1(r,\nu) \Vert x - z \Vert,
\end{eqnarray*}
where $(n,x), (n,z) \in \mathbb{Z} \times B_r(0)$, and $M_1(r,\nu) \rightarrow 0$ and $\Vert f_{\nu} \Vert \rightarrow 0$ as $\nu \rightarrow 0$ for fixed $r$.
\end{enumerate}

Thus, we have the following result:  

\begin{thm}\label{thm_difference_perturbed}
Let $\xi(n)$ be the unique remotely almost-periodic solution of the non-homogeneous linear difference equation \eqref{lineal+h_discreta},  
and suppose that conditions \ref{D1} and \ref{D2} are satisfied.  
Then, given $r$, there exists $\nu_0 = \nu_0(r) > 0$ sufficiently small such that the system \eqref{difference_perturbed}  
has a unique remotely almost-periodic solution $\psi_\nu(n)$ in an $r$-neighborhood of $\xi(n)$, for each fixed $\nu \in [0, \nu_0]$.  
Moreover, if $f_\nu(n, z)$ is uniformly continuous in $t \in \mathbb{R}$ with $\nu \in [0, \nu_0(r)]$,  
then $\psi_\nu$ is continuous in $\nu$, and we have  
\[
\lim_{\nu \rightarrow 0} \psi_\nu(n) = \xi(n).
\]
\end{thm}

\begin{proof}
Let $z(n) = y(n) - \xi(n)$, so we have
\begin{eqnarray*}
z(n+1)&=&y(n+1)-\xi(n+1)\\
&=& C(n)y(n)+h(n)+f_{\nu}(n,y(n))-C(n)\xi(n)-h(n)\\
z(n+1)&=& C(n)z(n)+f_{\nu}(n,z(n)+\xi(n))
\end{eqnarray*}
Let us consider
\begin{eqnarray*}
\tilde{B}(r)= \{\varphi_{\nu}\in \mbox{RAP}(\mathbb{Z},\bb{R}^q) \vert \Vert \varphi_{\nu }\Vert_{\infty}\leq r \}.
\end{eqnarray*}
Then, if we consider  
\begin{eqnarray*}
z(n+1) &=& C(n)z(n) + f_{\nu}(n, \varphi_{\nu}(n) + \xi(n)). \label{perturbadaaa}
\end{eqnarray*}  
By Theorem \ref{difference_sol_bounded}, the remotely almost-periodic solution is given by
\begin{eqnarray*}
T\varphi_{\nu}(n)=\sum_{k\in \bb{Z}} \tilde{G}(n,k)f_{\nu}(k,\varphi_{\nu}(k)+\xi(k)),\, \, n\in \bb{Z}.
\end{eqnarray*}
Moreover, we have that  
\[
\Vert T\varphi_{\nu}(n) \Vert \leq \sum_{k \in \mathbb{Z}} K e^{-\alpha |n - k|} \Vert f_{\nu} \Vert_{\infty}
\]  
We choose $\nu_1$ such that  
\[
\sum_{k \in \mathbb{Z}} K e^{-\alpha |n - k|} \Vert f_{\nu} \Vert_{\infty} < r.
\]  

Let $\varphi_{\nu}, \psi_{\nu} \in \tilde{B}(r)$, we have
\begin{eqnarray*}
\Vert T\varphi_{\nu}(n)-T\psi_{\nu}(n)\Vert\leq M_{1}(r,\nu)\sum_{k\in\bb{Z}}Ke^{-\alpha\vert n-k\vert}\Vert\varphi_{\nu}-\psi_{\nu}\Vert_{\infty}
\end{eqnarray*} 
We choose $\nu_2$ and $r$ sufficiently small such that  
\[
M_1(r,\nu) \sum_{k \in \mathbb{Z}} K e^{-\alpha |n - k|} < 1.
\]  
For $\nu_0 = \min\{\nu_1, \nu_2\}$, we obtain that $T: \tilde{B} \rightarrow \tilde{B}$ is a contractive operator. Then, there exists a unique fixed point $\varphi_{\nu}(n) \in \tilde{B}$ such that  
\[
T\varphi_{\nu}(n) = \varphi_{\nu}(n).
\]  
Here, $\varphi_{\nu}(n) = y(n) - \xi(n)$ is the unique remotely almost-periodic solution of \eqref{perturbadaaa},  
and thus $\psi_{\nu}(n) = \varphi_{\nu}(n) + \xi(n)$ is a solution of \eqref{difference_perturbed},  
moreover it satisfies $\left| \psi_{\nu}(n) - \xi(n) \right| \leq r$. This finishes the proof
\end{proof}

\section{Existence of Remotely Almost Periodic Solutions for DEPCA}
\label{sec5}
In what follows, $A$ and $B$ will be remotely almost periodic matrices, and the function $f$ will be remotely almost periodic unless stated otherwise.  
We will assume  
\[
\max\left\{ \left\Vert A \right\Vert_{\infty}, \left\Vert B \right\Vert_{\infty}, \left\Vert f \right\Vert_{\infty} \right\} \leq M.
\]

\begin{lem}\label{coef_ecua_diferencia}
Let $A$ be a remotely almost periodic matrix, $f(t) \in \mathbb{Z}RAP(\mathbb{R}, \mathbb{R}^q)$, and let $\phi$ be a fundamental matrix of \eqref{DEPCA_A} and defining $\Phi(t,s):=\phi(t)\phi^{-1}(s)$ for all $t,s\in\mathbb{R}$. Then the following holds:  
\begin{enumerate}
\item \label{Phi_discreta} The matrix $\Phi(n+1, n)$ is a discrete remotely almost periodic function.
\item \label{int_discreta} The sequence $h(n) = \int_{n}^{n+1} \Phi(n+1, u) f(u)\, du$ is a discrete remotely almost periodic function.
\end{enumerate}
\end{lem}
\begin{proof}
\eqref{Phi_discreta} This follows from \cite[Theorem 2]{pinto_remotely} with $t = n + 1$, $s = n$.

\eqref{int_discreta} We have that

\begin{eqnarray*}
\left\Vert h(n+\tau)-h(n)\right\Vert & = & \left\Vert \int_{n}^{n+1}\Phi(n+\tau+1,u+\tau)f(u+\tau)du-\int_{n}^{n+1}\Phi(n+1,u)f(u)du\right\Vert\\
 & \leq & \int_{n}^{n+1}\left\Vert\Phi(n+\tau+1,u+\tau)f(u+\tau)-\Phi(n+1,u)f(u)\right\Vert du\\
 & \leq & \int_{n}^{n+1}\left\Vert \Phi(n+\tau+1,u+\tau)-\Phi(n+1,u)\right\Vert \left\Vert f(u+\tau)\right\Vert du\\
 &  & +\int_{n}^{n+1}\left\Vert \Phi(n+1,u)\right\Vert \left\Vert f(u+\tau)-f(u)\right\Vert du\\
 & \leq & M\sup_{u\in[n,n+1]}\left\Vert \Phi(n+\tau+1,u+\tau)-\Phi(n+1,u)\right\Vert\\
 &  & +k_{0}\sup_{u\in[n,n+1]}\left\Vert f(u+\tau)-f(u)\right\Vert.
\end{eqnarray*}
By \cite[Theorem 2]{pinto_remotely}, we have that for $\epsilon > 0$, there exists $\epsilon'(\epsilon) = \epsilon'$ such that if $\tau \in T(A, \epsilon'/M)$, then
\begin{eqnarray*}
\limsup_{\vert n \vert\rightarrow \infty} M\Vert \Phi(n+\tau+1, u+\tau)-\Phi(n+1,u) \Vert\leq \epsilon \,\,\,\,\mathrm{ if }\,\,\,\, 0\leq \vert n+1-u\vert <1.
\end{eqnarray*}
Then, given that
$$\sup_{_{u\in[n,n+1]}} \left\Vert f(u+\tau)-f(u)\right\Vert \leq \sup_{u\in[n,\infty)} \left\Vert f(u+\tau)-f(u)\right\Vert$$
we have
\begin{eqnarray*}
\limsup_{ n \rightarrow \infty}\sup_{u\in[n,n+1]}\left\Vert f(u+\tau)-f(u)\right\Vert &\leq&  \limsup_{ n \rightarrow \infty}\sup_{u\in[n,\infty)}\left\Vert f(u+\tau)-f(u)\right\Vert\\
&=& \limsup_{n \rightarrow \infty} \left\Vert f(n+\tau)-f(n)\right\Vert.
\end{eqnarray*}
Let $\epsilon'' = \min\left\{ \frac{\epsilon'}{2M}, \frac{\epsilon}{2k_0} \right\}$, where $\epsilon' = \frac{\epsilon}{M + k_0}$.  
Consider $\tau \in T(\Phi, \epsilon'') \cap T(f, \epsilon'') \cap \mathbb{Z}$, then for $n\rightarrow\infty$
\begin{eqnarray*}
\limsup_{ n\rightarrow\infty}\left\Vert \int_{n+\tau}^{n+\tau+1}\Phi(n+\tau+1,u)f(u)du-\int_{n}^{n+1}\Phi(n+1,u)f(u)du\right\Vert \leq \epsilon.
\end{eqnarray*}
The situation for $n \rightarrow -\infty$ is analogous.  
This concludes the proof.
\end{proof}

The following lemma will help us relate $\mathbb{Z}RAP$ solutions to $RAP$ solutions in equation \eqref{DEPCA_lineal+f}.  
\begin{lem}\label{sol_DEPCA_uniform_continua}
Let $A, B$ and $f$ be bounded and locally integrable functions, then every bounded solution of \eqref{DEPCA_lineal+f} is uniformly continuous.
\end{lem}

\begin{proof}
Let $y$ be a bounded solution of \eqref{DEPCA_lineal+f}.  
Since $A, B$ and $f$ are bounded, there exists $M_0 > 0$ such that  
\[
\sup_{t \in \mathbb{R}} \left\Vert A(t)y(t) + B(t)y([t]) + f(t) \right\Vert < M_0.
\]  
Then, by the Fundamental Theorem of Calculus,  
\begin{eqnarray*}
\Vert y(t) - y(s) \Vert &\leq& \left\Vert \int_{s}^{t} \left( A(u)y(u) + B(u)y([u]) + f(u) \right) du \right\Vert \leq M_0 |t - s|,\quad\quad t,s\in\mathbb{R}.
\end{eqnarray*}  
Thus, $y$ is a Lipschitz function, and in particular, it is uniformly continuous.
\end{proof}

\subsection{Proof Theorem \ref{existencia_solucion_depca_rap}} Thus, we can state one of the main results of this chapter.

\begin{proof}
Using the variation of parameters formula on the interval $[n, n+1)$, we know that the solution of \eqref{DEPCA_lineal+f} satisfies  
\begin{eqnarray*}
\tilde{y}(t) = \left[\Phi(t,n) + \int_{n}^{t} \Phi(t,u) B(u)\,du \right] \tilde{y}(n) + \int_{n}^{t} \Phi(t,u) f(u)\,du, \quad t \in \mathbb{R},\ n = [t],\ n \leq t < n+1.
\end{eqnarray*}  
From the above expression, we obtain the non-homogeneous difference equation \eqref{lineal+h_discreta}, with $C(n)$ and $h(n)$ given by \eqref{C(n)} and \eqref{h(n)}, respectively.  

By Lemma \ref{coef_ecua_diferencia}, we can apply Theorem \ref{difference_sol_bounded}, which allows us to conclude that the non-homogeneous difference equation \eqref{lineal+h_discreta} has a unique remotely almost periodic solution $\{y(n)\}_{n \in \mathbb{Z}}$.  
Moreover, $|y(n)| \leq \gamma$, for all $n \in \mathbb{Z}$.

We now show that the solution $\tilde{y}(t)$ of \eqref{DEPCA_lineal+f} with $\tilde{y}(n) = y(n)$, $n \in \mathbb{Z}$, is remotely almost periodic.

If $\tau \in \mathbb{Z}$, we know that $[t + \tau] = [t] + \tau$. We have

\[
\begin{aligned}
\Vert  \tilde{y} &(t+\tau)-\tilde{y}(t) \Vert 
\\
\leq &~{} \Vert \left[ \Phi(t+\tau, [t]+\tau)-\Phi(t,[t]) \right] y([t]+\tau) \Vert 
+\Vert \Phi(t,[t]) \left[y([t]+\tau)-y([t]) \right]\Vert \\
\ & +\left\Vert \int_{[t]}^{t} \left(\Phi(t+\tau,u+\tau)B(u+\tau)-\Phi(t,u)B(u)\right)du\ y([t]+\tau)\right\Vert\\
\ & + \left\Vert \left(y([t]+\tau)-y([t])\right)\int_{[t]}^{t} \Phi(t,u)B(u)du\right\Vert\\
\ & +\left\Vert \int_{[t]}^{t} \left(\Phi(t+\tau, u+\tau)f(u+\tau)-\Phi(t,u)f(u)\right)du\right\Vert \\
\leq&~{} \Vert\Phi(t+\tau,[t]+\tau)-\Phi(t,[t])\Vert \left\Vert y([t]+\tau)\right\Vert 
+\left\Vert \Phi(t,[t])\right\Vert \Vert y([t]+\tau)-y([t])\Vert \\
&+\left\Vert y([t]+\tau)-y([t])\right\Vert \int_{[t]}^{t}\left\Vert \Phi(t,u)B(u)\right\Vert du\\
&+\int_{[t]}^{t}\left\Vert \Phi(t+\tau,u+\tau)\right\Vert \left\Vert B(u+\tau)-B(u)\right\Vert du\left\Vert y([t]+\tau)\right\Vert \\
&+\int_{[t]}^{t}\left\Vert \Phi(t+\tau,u+\tau)-\Phi(t,u)\right\Vert \left\Vert B(u)\right\Vert du\left\Vert y([t]+\tau)\right\Vert \\
&+\int_{[t]}^{t}\left\Vert \Phi(t+\tau,u+\tau)\right\Vert \left\Vert f(u+\tau)-f(u)\right\Vert du
\\&+\int_{[t]}^{t}\left\Vert \Phi(t+\tau,u+\tau)-\Phi(t,u)\right\Vert \left|f(u)\right\Vert du\\
\end{aligned}
\]
By \cite[Theorem 2]{pinto_remotely}, we have that $(t,s)\mapsto\Phi(t, s)$ is locally bi-remotely almost periodic,  
the sequence $\{y(n)\}_{n \in \mathbb{Z}}$ is remotely almost periodic,  
and by hypothesis, $B$ and $f$ are remotely almost periodic and $\mathbb{Z}$-remotely almost periodic functions, respectively.

Since $f$ is a $\mathbb{Z}$-remotely almost periodic function, if $\tau \in T(f, \epsilon)$, we have

\begin{eqnarray*}
\limsup_{t\rightarrow\infty} \sup_{u\in[\,[t],t]}\left\Vert f(u+\tau)-f(u)\right\Vert
&\leq&
\limsup_{t\rightarrow\infty} \sup_{u\in[t,\infty)}\left\Vert f(u+\tau)-f(u)\right\Vert \\
&=& \limsup_{t\rightarrow\infty} \left\Vert f(t+\tau)-f(t)\right\Vert <\epsilon
\end{eqnarray*}
Let $\epsilon' = \epsilon(\gamma + k_0 + k_0 M + \beta k_0 + M \beta + k_0 + M)^{-1}$.  
If $\tau \in T(y, \epsilon') \cap T(A, \epsilon') \cap T(B, \epsilon') \cap T(f, \epsilon') \cap \mathbb{Z}$,  
and recalling that $\Vert \Phi(t,s) \Vert < k_0$ when $0 \leq |t - s| \leq 1$, then

\begin{eqnarray*}
\Vert \tilde{y}(t+\tau)-\tilde{y}(t) \Vert  &\leq& \Vert\Phi(t+\tau,[t]+\tau)-\Phi(t,[t])\Vert\gamma+k_{0}\Vert y([t]+\tau)-y([t])\Vert\\& &+\left\Vert y([t]+\tau)-y([t])\right\Vert k_{0}M+\gamma k_{0}\sup_{u\in[\,[t],t]}\left\Vert B(u+\tau)-B(u)\right\Vert \\
&&+M\gamma\sup_{u\in[\,[t],t]}\left\Vert \Phi(t+\tau,u+\tau)-\Phi(t,u)\right\Vert +k_{0}\sup_{u\in[\,[t],t]}\left\Vert f(u+\tau)-f(u)\right\Vert \\
&&+M\sup_{u\in[\,[t],t]}\left\Vert \Phi(t+\tau,u+\tau)-\Phi(t,u)\right\Vert .
\end{eqnarray*}
We can conclude that
\begin{eqnarray*}
\limsup_{\vert t\vert\rightarrow\infty}\Vert\tilde{y}(t+\tau)-\tilde{y}(t)\Vert&\leq&\epsilon,
\end{eqnarray*}
that is, the solution $\tilde{y}$ is $\mathbb{Z}$-remotely almost periodic and is unique,  
since the sequence $\{y(n)\}_{n \in \mathbb{Z}}$ uniquely determines the solution of the differential equation with piecewise constant argument.

By Lemmas \ref{ZRAP_unif_conti_RAP} and \ref{sol_DEPCA_uniform_continua}, we conclude that the solution is remotely almost periodic.
\end{proof}

\section{Existence of Remotely Almost Periodic Solutions for Perturbed DEPCA}\label{sec6}

In \cite{xia_3}, Xia et al. consider the following almost periodic systems with piecewise constant argument that contain a small parameter $\nu$, of the form  
\begin{eqnarray*}
x'(t) &=& A(t)x(t) + B(t)x([t]) + f(t) + \nu g(t, x(t), x([t]), \nu), \\
x'(t) &=& \tilde{f}(t, x(t), x([t])) + \nu g(t, x(t), x([t]), \nu),
\end{eqnarray*}  
forall $t\in\mathbb{R}$. Under certain sufficient conditions, they obtain the existence of almost periodic solutions to these systems.  
This is one of the few works that studies equations of this type.

Motivated by these works, we will study the existence of remotely almost periodic solutions in differential equations with piecewise constant argument,  
correcting existing errors in the literature, such as those in \cite{xia_3}.

In this section, we will study non-homogeneous differential equations with piecewise constant argument of the form

\begin{eqnarray}
y'\left(t\right)&=&A\left(t\right)y\left(t\right)+B\left(t\right)y\left(\left[t\right]\right)+f\left(t\right)+ g_{\nu}\left(t,y\left(t\right),y\left(\left[t\right]\right)\right),\,\,\, t\in\mathbb{R},\label{depca_lineal+f_perturbada}\\
z'\left(t\right)&=&\widetilde{f}\left(t,z\left(t\right),z\left(\left[t\right]\right)\right)+ g_{\nu}\left(t,z\left(t\right),z\left(\left[t\right]\right)\right),\,\,\, t\in\mathbb{R},\label{depca_nolineal_perturbada}
\end{eqnarray}
where $A,B:\mathbb{R}\rightarrow\mathbb{R}^{q\times q}$, $f:\mathbb{R}\rightarrow\mathbb{R}^{q}$,$\,\tilde{f},g_{\nu}:\mathbb{R}\times\mathbb{R}^{q}\times\mathbb{R}^{q}\rightarrow\mathbb{R}^{q}$, moreover, $g_\nu(t, x_1, x_2) = g(t, x_1, x_2, \nu)$ such that $g(t, x_1, x_2, 0) = 0$. \\  
We will use exponential dichotomy and the contraction principle.  
Under certain conditions, we will obtain the existence and uniqueness of remotely almost periodic solutions $x_\nu$,  
such that $\lim_{\nu \rightarrow 0} x_\nu$ exists and corresponds to the unperturbed case.

Let us consider the differential equations with piecewise constant argument \eqref{DEPCA_lineal}, \eqref{DEPCA_lineal+f}, and \eqref{depca_lineal+f_perturbada},  
where $A, B : \mathbb{R} \rightarrow \mathbb{R}^{q \times q}$, $f : \mathbb{R} \rightarrow \mathbb{R}^q$, and $g_\nu : \mathbb{R} \times \mathbb{R}^q \times \mathbb{R}^q \rightarrow \mathbb{R}^q$  
are continuous, and $\nu \in [0, \nu_0]$. \\  
In what follows, we will assume that  
\[
\max \left\{ \Vert A \Vert_{\infty}, \Vert B \Vert_{\infty}, \Vert f \Vert_{\infty} \right\} \leq M.
\]

Let us consider the following hypotheses:  
\begin{enumerate}[label=\textbf{(H.\arabic*)}]
\item\label{H1} Let $A(t)$, $B(t)$, and $f(t)$ be remotely almost periodic functions.  
The discrete equation associated with \eqref{DEPCA_lineal} admits an exponential dichotomy such that its associated Green's kernel is bi-remotely almost periodic and summable.

\item\label{H2} Let $g_{\nu}(t, y_1, y_2)$ be a remotely almost periodic function in $t$  
uniformly for $(y_1, y_2, \nu) \in B_r(0) \times B_r(0) \times [0, \nu_0]$,  
and for each fixed small real parameter $\nu$, it is uniformly bounded.  
It satisfies the local Lipschitz condition  
\[
\left\Vert g_{\nu}(t, x_1, y_1) - g_{\nu}(t, x_2, y_2) \right\Vert \leq M_0(r, \nu) \left[ \left\Vert x_1 - x_2 \right\Vert + \left\Vert y_1 - y_2 \right\Vert \right],
\]
where $(t, x_1, y_1), (t, x_2, y_2) \in \mathbb{R} \times B_r(0) \times B_r(0)$ and $\nu \in [0, \nu_0]$,  
such that $M_0(r, \nu) \to 0$ and  
\[
\Vert g_{\nu} \Vert_r = \sup_{t \in \mathbb{R},\, \Vert x \Vert, \Vert y \Vert \leq r} \Vert g_\nu(t, x, y) \Vert \to 0
\]  
as $\nu \to 0$, for every fixed $r$.
\end{enumerate}

\subsection{Proof Theorem \ref{teo_perturbada_depca} }
Let $x\left(t\right)=y\left(t\right)-\xi\left(t\right);$ we have
\[
\begin{aligned}
x^{\prime}(t)  = &~{} y^{\prime}(t)-\xi^{\prime}(t)\nonumber \\
  = &~{} A\left(t\right)x\left(t\right)+B\left(t\right)x\left(\left[t\right]\right)+ g_{\nu}\left(t,x\left(t\right)+\xi\left(t\right),x\left(\left[t\right]\right)+\xi\left(\left[t\right]\right)\right).\label{eq:3.4}
\end{aligned}
\]
We denote 
\[
\begin{aligned}
B & =B\left(r,\nu\right)
\\=&  \big\{ \varphi\left(t,\nu\right)=\varphi_{\nu}(t)\in C\left(\mathbb{R}\times\left[0,\nu_{0}\right],\mathbb{R}^{q}\right) \big|\, RAP\mbox{ in } 
 t\,\mbox{ for every }\,\nu\in\left[0,\nu_{0}\right]\mbox{ fixed },\left|\varphi_{\nu}\left(t\right)\right|\leq r\big\} 
\end{aligned}
\]
which is a complete metric subspace of $RAP(\bb{R},\mathbb{R}^{q})$.

For each $\nu\in\left[0,\nu_{0}\right]$ fixed, let $\varphi_{\nu}\in B$, applying the Theorem \ref{existencia_solucion_depca_rap}, the following equation
\begin{equation}
x'\left(t\right)=A\left(t\right)x\left(t\right)+B\left(t\right)x\left(\left[t\right]\right)+g_{\nu}\left(t,\varphi_{\nu}\left(t\right)+\xi\left(t\right),\varphi_{\nu}\left(\left[t\right]\right)+\xi\left(\left[t\right]\right)\right),\label{eq:operator}
\end{equation}
has a unique remotely almost periodic solution $T\varphi_{\nu}\left(t\right)$.

Using the variation of parameters formula on the interval $[n,n+1)$, the solution $T\varphi_{\nu}(t)$ can be written as
\begin{equation}
  \begin{aligned}
T\varphi_{\nu}\left(t\right)  = & \left[\Phi\left(t,n\right)+\int_{n}^{t}\Phi\left(t,u\right)B\left(u\right)du\right]T\varphi_{\nu}\left(n\right) \\
  & + \int_{n}^{t}\Phi\left(t,u\right)g_{\nu}\left(u,\varphi_{\nu}\left(u\right)+\xi\left(u\right),\varphi_{\nu}\left(n\right)+\xi\left(n\right)\right)du,\label{operadoooor}
\end{aligned}
\end{equation}
$t\in\mathbb{R},\, n=\left[t\right],\, n\leq t< n+1.$\\
Moreover, the sequence $\left\{ T\varphi_{\nu}(n) \right\}_{n \in \mathbb{Z}}$ is the unique remotely almost periodic solution of the non-homogeneous difference equation associated with \eqref{eq:operator}, that is, of
\begin{equation*}
y\left(n+1\right)=C\left(n\right)y\left(n\right)+ h_{\nu}\left(n\right)\label{eq:difference.equationsistem}
\end{equation*}
where
\begin{eqnarray*}
h_{\nu}\left(n\right) & = & \int_{n}^{n+1}\Phi\left(n+1,u\right)g_{\nu}\left(u,\varphi_{\nu}\left(u\right)+\xi\left(u\right),\varphi_{\nu}\left(n\right)+\xi\left(n\right)\right)du.
\end{eqnarray*}
and $C$ is given by \eqref{C(n)}, which are remotely almost periodic by Lemma \ref{coef_ecua_diferencia}.

By Theorem \ref{difference_sol_bounded}, the sequence satisfies the inequality
\begin{eqnarray}
\left\Vert T\varphi_{\nu}\left(n\right)\right\Vert \leq K\left(1+e^{-\alpha}\right)\left(1-e^{-\alpha}\right)^{-1}\left\Vert h_{\nu}\right\Vert_{\infty},\,\forall n\in\mathbb{Z}.\label{cota_sucesion}
\end{eqnarray}
Now, we show $T\varphi_{\nu}\in B$. First we show $\left\Vert T\varphi_{\nu}\left(t\right)\right\Vert \leq r.$

By \ref{H2}, we have $\left\Vert g_{\nu}\left(t,x_{1},y_{1}\right)\right\Vert_{\tilde{r}} \rightarrow0$
when $\nu\rightarrow0$ for $\left(t,x_{1},y_{1}\right)\in\mathbb{R}\times B_{\tilde{r}}(0)\times B_{\tilde{r}}(0)$ con  $\tilde{r}=r+\Vert \xi \Vert_{\infty}$, then we choose $\nu_{1}\left(r\right)$ small enough such that
\begin{equation}
2K_{0}\left[K_{0}\left(M+1\right)K\left(1+e^{-\alpha}\right)\left(1-e^{-\alpha}\right)^{-1}+1\right] \left\Vert g_{\nu}\right\Vert_{\tilde{r}} \leq r,\,\,\,\mbox{for }\nu\in\left[0,\nu_{1}\left(r\right)\right],\label{3.9}
\end{equation}
where  $K_0$ is such that $\Vert \Phi(t,s) \Vert<K_0$ for $0\leq \vert t-s\vert \leq 1$.

Note that the Lipschitz constant $M_{0}(r,\nu)$ can be chosen as a non-decreasing function in $\nu$.  
Moreover, $M_{0}(\tilde{r},\nu) \rightarrow 0$ as $\nu \rightarrow 0$ for fixed $r$,  
so there exists $\nu_0(r)$ sufficiently small such that $\nu_0(r) \leq \nu_1$, and furthermore

\begin{equation}
2K_{0}\left[K_{0}\left(M+1\right)K\left(1+e^{-\alpha}\right)\left(1-e^{-\alpha}\right)^{-1}+1\right]M_{0}\left(\tilde{r},\nu\right)<1,\,\,\,\mbox{for }\nu\in\left[0,\nu_{0}(r)\right].\label{eq:3.8}
\end{equation}
Note that given $n\in\mathbb{Z}$
\begin{equation}\label{cota_h}
\begin{aligned}
\left\Vert h_{\nu}\left(n\right)\right\Vert   = & \left|\int_{n}^{n+1}\Phi\left(n+1,u\right)g_{\nu}\left(u,\varphi_{\nu}\left(u\right)+\xi\left(u\right),\varphi_{\nu}\left(n\right)+\xi\left(n\right)\right)du\right|\\
  \leq & K_{0}\left\Vert g_{\nu}\right\Vert_{\tilde{r}} .
\end{aligned}
\end{equation}
Then for each $\varphi_{\nu}\in B,\,\nu\in\left[0,\nu_{0}\left(r\right)\right],$
it follows that \eqref{operadoooor}, \eqref{cota_sucesion}, \eqref{3.9}   and \eqref{cota_h} the following
\begin{eqnarray*}
\left\Vert T\varphi_{\nu}\left(t\right)\right\Vert 
&\leq& K_{0}\left(1+M\right)\left\Vert T\varphi_{\nu}\left(n\right)\right\Vert +K_{0}\left\Vert g_{\nu}\right\Vert _{\tilde{r}} \nonumber \\
&\leq& 2K_{0}\left[K_{0}\left(M+1\right)K\frac{\left(1+e^{-\alpha}\right)}{\left(1-e^{-\alpha}\right)}+1\right]\left\Vert g_{\nu}\right\Vert_{\tilde{r}} \leq r,\,\,\mbox{for }\nu\in\left[0,\nu_{0}\left(r\right)\right]\label{eq:3.11}
\end{eqnarray*}

Therefore $T\varphi_{\nu}\in B$.

Finally, we will prove that $T$ is a contractive operator.  
Indeed, for each $\psi_{\nu}, \varphi_{\nu} \in B$, $\nu \in [0, \nu_0(r)]$, we have

\begin{eqnarray*}
T\varphi_{\nu}\left(t\right)-T\psi_{\nu}\left(t\right)
&= &  \left[\Phi\left(t,n\right)+\int_{n}^{t}\Phi\left(t,u\right)B\left(u\right)du\right]\left[T\varphi_{\nu}\left(n\right)-T\psi_{\nu}\left(n\right)\right]\\
 & & + \int_{n}^{t}\Phi\left(t,u\right)\left[g_{\nu}\left(u,\varphi_{\nu}\left(u\right)+\xi\left(u\right),\varphi_{\nu}\left(n\right)+\xi\left(n\right)\right)\right.\\
 & & \left. -g_{\nu}\left(u,\varphi\left(u\right)+\xi\left(u\right),\varphi_{\nu}\left(n\right)+\xi\left(n\right)\right)\right]du,\,\,\,\, 
\end{eqnarray*}
where $n\leq t<n+1$. It follows
\[
\Vert T\varphi_{\nu}-T\psi_{\nu}\Vert_{\infty} \leq 2K_{0}\left[K_{0}\left(M+1\right)K\frac{\left(1+e^{-\alpha}\right)}{\left(1-e^{-\alpha}\right)}+1\right] M_{0}\left(\tilde{r},\nu\right)\Vert\varphi_{\nu}-\psi_{\nu}\Vert_{\infty}.
\]
Then, by \eqref{eq:3.8}, we obtain that $T: B \rightarrow B$ is a contractive operator.  
Therefore, by the Banach fixed point theorem, there exists a unique fixed point $\varphi_{\nu}(t) \in B$  
such that $T\varphi_{\nu}(t) = \varphi_{\nu}(t).$

\medskip

Given that $\varphi_{\nu}=x-\xi$ is the unique remotely almost periodic solution of \eqref{eq:operator},
then $\psi_{\nu}=\varphi_{\nu}+\xi$ is solution of \eqref{depca_lineal+f_perturbada}, also it satisfies $\left\Vert \psi_{\nu}\left(t\right)-\xi\left(t\right)\right\Vert \leq r$ for all $t\in\mathbb{R}$.

Lastly, we will see that for each $t\in\mathbb{R}$ we have  $\lim_{\nu\rightarrow0}\psi_{\nu}\left(t\right)=\xi(t)$. To this end, let us note that
\begin{eqnarray*}
\lim_{\nu\rightarrow 0} \Vert h_\nu(n)\Vert =0 \quad\quad \forall n \in \mathbb{Z}.
\end{eqnarray*}
Then by Theorem \ref{thm_difference_perturbed}, we have 
$$\lim_{\nu\rightarrow 0}\Vert \varphi_{\nu}(n)\Vert=0, \quad\quad \forall n \in \mathbb{Z}.$$
Finally,
\begin{eqnarray*}
\lim_{\nu\rightarrow 0} \Vert \varphi_{\nu}(t) \Vert &\leq &\lim_{\nu\rightarrow 0} 
\left\Vert \Phi\left(t,n\right)+\int_{n}^{t}\Phi\left(t,u\right)B\left(u\right)du\right\Vert \Vert\varphi_{\nu}\left(n\right)\Vert\\
 &  & + \lim_{\nu\rightarrow 0} \int_{n}^{t}\Vert \Phi\left(t,u\right)\Vert \Vert g_{\nu}\left(u,\varphi_{\nu}\left(u\right)+\xi\left(u\right),\varphi_{\nu}\left(n\right)+\xi\left(n\right)\right)\Vert du\\
 &=& 0,
\end{eqnarray*}
given that
\begin{eqnarray*}
\lim_{\nu\rightarrow 0} \Vert\varphi_{\nu}\left(n\right)\Vert =0\ \ \ \quad\mbox{ and }\quad \ \ \lim_{\nu\rightarrow 0}\Vert g_{\nu}\left(u,\varphi_{\nu}\left(u\right)+\xi\left(u\right),\varphi_{\nu}\left(n\right)+\xi\left(n\right)\right)\Vert =0
\end{eqnarray*}
we can conclude that
\begin{eqnarray*}
\lim_{\nu\rightarrow 0} \psi_{\nu}(t)=\xi(t) \quad \quad \forall t\in\mathbb{R},
\end{eqnarray*}
which it finishes the proof.

We now study the perturbed system:
\begin{equation}
x^{\prime}(t) = A_{\nu}(t)x(t) + B_{\nu}(t)x([t]) + f(t) + g_{\nu}(t, x(t), x([t])) \label{eq:A_v(t)}
\end{equation}
where $A_{\nu}(t), B_{\nu}(t)$ are square matrices of order $q$,  
remotely almost periodic and defined on $\mathbb{R}$, with $\nu \in [0, \nu_0]$.  
Moreover, $A_{\nu} \rightrightarrows A_{0}$ and $B_{\nu} \rightrightarrows B_{0}$ as $\nu \to 0$.  
The functions $f$ and $g$ are as in the previous theorem.

In addition, we will consider the following systems
\begin{eqnarray}
y^{\prime}(t)&=&A_{0}\left(t\right)y(t)+B_{0}(t)y([t])\label{eq:A_0}\\
z^{\prime}(t)&=&A_{0}\left(t\right)z(t)+B_{0}(t)z([t])+f\left(t\right)\label{eq:A_0+f}
\end{eqnarray}
where $A_0(t)$, $B_0(t)$ are (matrix-valued) remotely almost periodic functions defined on $\mathbb{R}$.  

Consider the following hypothesis:  

\begin{enumerate}[label=\textbf{(H'.\arabic*)}]
\item\label{H1prime} Let $\xi(t)$ be the unique remotely almost periodic solution of \eqref{eq:A_0+f}.  
The system \eqref{eq:A_0} satisfies hypothesis \ref{H1}.  
Moreover, suppose that $A_\nu \rightrightarrows A_0$ and $B_\nu \rightrightarrows B_0$ uniformly on $\mathbb{R}$ as $\nu \to 0$.
\end{enumerate}

\begin{cor}
Assume that \ref{H1prime} and \ref{H2} are satisfied.  
Then there exist $r > 0$ and $\nu_0 = \nu_0(r)$ sufficiently small such that the system \eqref{eq:A_v(t)}  
has a unique remotely almost periodic solution $\psi_\nu(t)$ in an $r$-neighborhood of $\xi(t)$,  
for each fixed $\nu \in [0, \nu_0]$.  
Moreover, if $g_\nu(t, x, y)$ is uniformly continuous for $(t, x, y) \in \mathbb{R} \times B_r(0) \times B_r(0)$, with $\nu \in [0, \nu_0]$,  
then $\psi_\nu(t)$ is continuous in $\nu$, and we have  
\[
\lim_{\nu \to 0} \psi_\nu(t) = \xi(t).
\]
\end{cor}
\begin{proof}
Let $u(t)=x(t)-\xi(t)$, we have
\[
\begin{aligned}
u'(t)	=&{}~	A_{0}\left(t\right)u(t)+B_{0}(t)u([t])+G_\nu (t,u(t),u([t]))
\end{aligned}
\]
where
\[
\begin{aligned}
G_\nu (t,u(t),u([t])) = &~{} g_{\nu}\left(t,\xi(t)+u(t),\xi([t])+u([t])\right)
\\&+(A_{\nu}(t)-A_{0}(t))(u(t)+\xi(t))+(B_{\nu}([t])-B_{0}([t]))(u([t])+\xi([t])).
\end{aligned}\]

It is straightforward to verify that the hypotheses from Theorem \ref{teo_perturbada_depca} are satisfied.
This concludes the Corollary.
\end{proof}

Let us consider the system with delay
\begin{align}
\frac{dy}{dt} & =A\left(t\right)y+h\left(t\right)+\nu g\left(t,y\left(t\right),y\left(t-\alpha\right),y([t])\right),\,\,\,\alpha>0\,\mbox{fixed},\label{eq:reetardo}
\end{align}
also, $\xi$ is the unique remotely almost periodic solution of
\begin{equation*}
\frac{dz}{dt}=A\left(t\right)z+h\left(t\right)
\end{equation*}

Consider the following hypothesis:
\begin{enumerate}[label=\textbf{(H'.2)}]
\item\label{H2prime} Let $g\left(t,x,y,z\right)$ be remotely almost periodic in $t$ uniformly with respect to $\left(x,y,z\right)\in  B_r (0)\times B_r (0)\times B_r (0)$, and satisfies the Lipschitz condition

\[
\left\Vert g\left(t,x_1,x_2,x_3\right)-g\left(t,y_1,y_2,y_3\right)\right\Vert \leq M_{1}\left(r\right)\big(\left\Vert x_1-y_1\right\Vert +\left\Vert x_2-y_2\right\Vert+ \Vert x_3-y_3\Vert \big),\,
\]
\medskip

for $\left(t,x_1,x_2,x_3\right),\left(t,y_1,y_2,y_3\right)\in \bb{R}\times B_r (0)\times B_r (0)\times B_r (0)$, with $r>0$ fixed.
\end{enumerate}

By considering $g_{\nu}(t, x_1, x_2, x_3) = \nu g(t, x_1, x_2, x_3)$,  
it follows immediately that $\Vert g_{\nu} \Vert \rightrightarrows 0$ as $\nu \to 0$.  
We have the following result.

\begin{thm}
If the linear part of \eqref{eq:reetardo} admits an exponential dichotomy and its associated Green's kernel is bi-remotely almost periodic and integrable,  
$h(t)$ is remotely almost periodic, and condition \ref{H2prime} is satisfied.

Then, for every fixed $r > 0$, there exists $\nu_0 = \nu_0(r)$ sufficiently small such that the system \eqref{eq:reetardo}  
has a unique remotely almost periodic solution $\psi_\nu(t)$ in an $r$-neighborhood of $\xi(t)$,  
for every $\nu \in [0, \nu_0]$.  
Moreover, $\psi_\nu(t)$ is continuous in $\nu$, and we have  
\[
\lim_{\nu \rightarrow 0} \psi_\nu(t) = \xi(t).
\]
\end{thm}
\begin{proof}
Let $u(t)=y(t)-\xi(t)$, we have
\[
\begin{aligned}
 u'(t)	=	A(t)u+\nu g\left(t,\xi(t)+u(t),\xi\left(t-\alpha\right)+u\left(t-\alpha\right),\xi([t])+u([t])\right).
\end{aligned}
\]
Consider $\nu_0=\min\{\frac{2\alpha r}{K\Vert g\Vert_{\infty}},\frac{2KM_1(r)}{\alpha}\}$.
Let
\[
\tilde{B}=\left\{ \varphi_{\nu}\left(t\right)\left|\varphi_{\nu}\left(t\right)\in\mbox{RAP},\,\Vert \varphi_{\nu}\left(t\right)\Vert \leq r,\ \nu\in[0,\nu_0] \right.\right\}.
\] 
Take $\varphi_{\nu}\in\tilde{B}$, we have
\begin{equation}
\frac{du}{dt}=A\left(t\right)u+\nu g\left(t,\xi(t)+\varphi_{\nu}(t),\xi\left(t-\alpha\right)+\varphi_{\nu}\left(t-\alpha\right),\xi([t])+\varphi_\nu([t])\right).\label{blaaaa}
\end{equation}
Then the solution is 
\[
\begin{aligned}
T\varphi_{\nu}(t)=\nu \int_{-\infty}^{\infty} G(t,s)g\left(s,\xi(s)+\varphi_{\nu}(s),\xi\left(t-\alpha\right)+\varphi_{\nu}\left(s-\alpha\right),\xi([s])+\varphi_{\nu}([s])\right)ds.
\end{aligned}
\]
It is easy to see that $T: \tilde{B} \rightarrow \tilde{B}$, and moreover, it is contractive.  
Then, by the Banach fixed point theorem, there exists a unique fixed point $\varphi_\nu(t) \in B$  
such that $T\varphi_\nu(t) = \varphi_\nu(t)$.  
Since $\varphi_\nu(t) = x(t) - \xi(t)$ is the unique remotely almost periodic solution of \eqref{blaaaa},  
it follows that $\psi_\nu = \varphi_\nu(t) + \xi(t)$ is a solution of \eqref{eq:reetardo}, and moreover,  
it satisfies $\left\Vert \psi_\nu(t) - \xi(t) \right\Vert \leq r$.
\end{proof}

\begin{rem}
Note that the main difference between this theorem and the previous ones lies in the fact that, in this result, for a given constant $r>0$, there exists $\nu_0 > 0$.  
That is, the interdependence between $r$ and $\nu_0$ present in the previous theorems does not exist here.
\end{rem}

Finally, consider the following systems
\begin{align}
\frac{dx}{dt} & =f\left(t,x(t),x([t])\right)\label{eq:generadorF}\\
\frac{dy}{dt} & =f\left(t,y(t),y([t])\right)+g_{\nu}\left(t,y\left(t\right),y([t])\right),\label{eq:perturbadof}
\end{align}
Under the hypothesis

\begin{enumerate}[label=\textbf{(C.\arabic*)}]
\item\label{C1} The equation \eqref{eq:generadorF} has a unique remotely almost periodic solution $\xi(t)$.
\item\label{C2} $f(t, x, y) \in C^{2}$ in $x$ and $y$, and the second-order partial derivatives satisfy a Lipschitz condition in $x$ and $y$.
\item\label{C3} The discrete equation associated with the variational equation with piecewise constant argument
\begin{eqnarray*}
\frac{dz}{dt} = \frac{\partial f}{\partial x}(t, \xi(t), \xi([t])) z(t) + \frac{\partial f}{\partial y}(t, \xi(t), \xi([t])) z([t])
\end{eqnarray*}
admits an exponential dichotomy such that its associated Green's kernel is bi-remotely almost periodic and summable, where $\xi(t)$ is defined in \ref{C1}.  
Moreover, $\frac{\partial f}{\partial x}(t, x, y)$ and $\frac{\partial f}{\partial y}(t, x, y)$ are remotely almost periodic functions in the first variable.
\end{enumerate}

\begin{thm}\label{perturbed_f_nolineal}
If conditions \ref{C1}, \ref{C2} and \ref{C3}, along with \ref{H2} are satisfied,  
then there exist $r > 0$ and $\nu_0 = \nu_0(r)$ sufficiently small such that the system \eqref{eq:perturbadof}  
has a unique remotely almost periodic solution $\psi_\nu(t)$ in an $r$-neighborhood of $\xi(t)$, for all $\nu \in [0, \nu_0]$.

Moreover, if $g_\nu(t, x, z)$ is uniformly continuous on $(t, x, z) \in \mathbb{R} \times B[0, r] \times B[0, r]$, with $\nu \in [0, \nu_0]$,  
then $\psi_\nu(t)$ is continuous in $\nu$, and we have  
\[
\lim_{\nu \rightarrow 0} \psi_\nu(t) = \xi(t).
\]
\end{thm}

\begin{proof}
By differentiation, we know
\begin{equation}\label{calculo_aprox_f}
\begin{aligned}
f(t,x(t) &+\xi(t),y([t])+\xi([t]))-f(t,\xi(t),\xi([t])) 
\\
=&~{}\frac{\partial f}{\partial x}(t,\xi(t),\xi([t]))x(t) 
+\frac{\partial f}{\partial y}(t,\xi(t),\xi([t]))y([t]) \\
&+\frac{1}{2}(x(t)\cdot \nabla_{x}+y([t])\cdot \nabla_{y})^{2}f(t,\theta x(t)+\xi(t),\theta y([t])+\xi([t])) \ \ \ 
,
\end{aligned}
\end{equation}
where $\theta \in [0, 1)$, ``$\cdot$'' denotes the inner product, and ``$\nabla$'' is the Hamiltonian operator. \\  
Let $z(t) = y(t) - \xi(t)$. Then, from \eqref{eq:perturbadof} and \eqref{eq:generadorF}, we have that
\[
\begin{aligned}
\frac{dz}{dt}(t)=&~{}f(t,z(t)+\xi(t),z([t])+\xi([t]))-f(t,\xi(t),\xi([t]))\\
 & +\nu g\left(t,z(t)+\xi(t),z([t])+\xi([t])\right),
\end{aligned}
\]
by \eqref{calculo_aprox_f}, we have
\[
\begin{aligned}
\frac{dz}{dt}(t)=&\frac{\partial f}{\partial x}(t,\xi(t),\xi([t]))z(t)+\frac{\partial f}{\partial y}(t,\xi(t),\xi([t]))z([t])\\
 &+\frac{1}{2}(x(t)\cdot \nabla_{x}+y([t])\cdot \nabla_{y})^{2}f(t,\theta z(t)+\xi(t),\theta z([t])+\xi([t]))\\
 &+\nu g\left(t,z(t)+\xi(t),z([t])+\xi([t])\right).
\end{aligned}
\]
Let $A(t)=\frac{\partial f}{\partial x}(t,\xi(t),\xi([t])) $, $B(t)=\frac{\partial f}{\partial y}(t,\xi(t),\xi([t]))$ and 
\[
\begin{aligned}
F(t,z(t),z([t]))=&\frac{1}{2}(x(t)\cdot \nabla_{x}+y([t])\cdot \nabla_{y})^{2}f(t,\theta z(t)+\xi(t),\theta z([t])+\xi([t])),
\end{aligned}
\]
without loss of generality, we can assume that  
$\max \{ \Vert A \Vert_{\infty}, \Vert B \Vert_{\infty} \} \leq M$.  
We obtain the following system:
\[
\begin{aligned}
\frac{dz}{dt}(t)=A(t)z(t)+B(t)z([t])+F(t,z(t),z([t]))+\nu g\left(t,z(t)+\xi(t),z([t])+\xi([t])\right)
\end{aligned}
\]
By $(C_2)$, there are $r>0$ y $\nu^{0}$, $N_i(r)$, $i=1,2$ y $M_0(\nu_{0})$ such that
\[
\begin{aligned}
\left\Vert\frac{\partial^{2}f}{\partial x_{i}\partial x_{j}}(t,\theta z(t)+\xi(t),\theta z([t])+\xi([t]))\right\Vert\leq &~{} N_{1}(r)\\
\left\Vert\frac{\partial^{2}f}{\partial y_{i}\partial y_{j}}(t,\theta z(t)+\xi(t),\theta z([t])+\xi([t]))\right\Vert\leq &~{} N_{1}(r)\\
\left\Vert\frac{\partial^{2}f}{\partial x_{i}\partial y_{j}}(t,\theta z(t)+\xi(t),\theta z([t])+\xi([t]))\right\Vert\leq &~{} N_{1}(r), 
\end{aligned}
\]
for $i,j=1,2,\cdots,q,\ t\in\bb{R}$ y $\Vert z\Vert\leq r$ and 
\[
\begin{aligned}
\bigg\Vert \frac{\partial^{2}f}{\partial x_{i}\partial x_{j}}(t,\theta z(t)+\xi(t),\theta z([t])+\xi([t])) -\frac{\partial^{2}f}{\partial x_{i}\partial x_{j}} & (t,\theta\tilde{z}(t)+\xi(t),\theta\tilde{z}([t])+\xi([t]))\bigg\Vert \\
&\leq~{}2\theta N_{2}(r)\left\Vert z(t)-\tilde{z}(t)\right\Vert
\\
\bigg\Vert \frac{\partial^{2}f}{\partial y_{i}\partial y_{j}}(t,\theta z(t)+\xi(t),\theta z([t])+\xi([t]))-\frac{\partial^{2}f}{\partial y_{i}\partial y_{j}}&(t,\theta\tilde{z}(t)+\xi(t),\theta\tilde{z}([t])+\xi([t]))\bigg\Vert \\
&\leq ~{} 2\theta N_{2}(r)\big\Vert z(t)-\tilde{z}(t)\big\Vert
\\
\bigg\Vert\frac{\partial^{2}f}{\partial x_{i}\partial y_{j}}(t,\theta z(t)+\xi(t),\theta z([t])+\xi([t]))-\frac{\partial^{2}f}{\partial x_{i}\partial y_{j}}&(t,\theta\tilde{z}(t)+\xi(t),\theta\tilde{z}([t])+\xi([t]))\bigg\Vert 
\\
&\leq~{} 2\theta N_{2}(r)\left\Vert z(t)-\tilde{z}(t)\right\Vert
\end{aligned}
\]
where $i,j=1,2,\cdots,q,\ t\in\bb{R}$ and $\Vert z\Vert\leq r$, $N_2(r)$ is bounded and $N_1(r),M_0(\nu)$ can be chosen as non decreasing functions of $r$ and $\nu$ respectively. It follows
\begin{eqnarray*}
\Vert F(t,z(t),z([t]))\Vert \leq 3r^{2}N_1(r) \mbox{ for } \Vert z\Vert\leq r
\end{eqnarray*}
and
\begin{eqnarray*}
\Vert F(t,z(t),z([t]))-F(t,\tilde{z}(t),\tilde{z}([t]))\Vert \leq 6(qrN_1(r)+r^{2}N_2(r))\Vert z-\tilde{z}\Vert
\end{eqnarray*}
for $\Vert z\Vert, \Vert\tilde{z} \Vert\leq r$ and $t\in \bb{R}$.

Then, since the hypotheses of Theorem \ref{teo_perturbada_depca} are satisfied, the theorem follows.

\end{proof}

\section{Lasota-Wazewska Model}\label{sec7}

The Lasota-Wazewska Model is an autonomous differential equation of the form
\begin{eqnarray}
y'(t)=-\delta y(t)+pe^{-\gamma y(t-\tau )},\ t\geq 0. \label{lasota-wazewska}
\end{eqnarray}
It was used by Wazewska-Czyzewska and Lasota \cite{lasota-wazewska} to describe the survival of red blood cells in an animal's bloodstream.  
In this equation, $y(t)$ represents the number of red blood cells in the blood at time $t$,  
$\delta > 0$ is the probability of red blood cell death;  
$p, \gamma$ are positive constants related to the production of red blood cells per unit of time,  
and $\tau$ is the time required to produce red blood cells. \\  
Equation \eqref{lasota-wazewska} models various real-life situations; see \cite{liz-martinez}. \\  
In \cite{paper_lazota_depca}, a variation of this model is proposed by modifying the delay
\begin{eqnarray}
y'(t)=-\delta y(t)+pe^{-\gamma y([t])},\ t\geq 0. \label{lasota-wazewska_modificada}
\end{eqnarray}

We will study the following delayed model:  
\begin{eqnarray*}
y'(t) = -\delta(t)\, y(t) + p(t)\, f(y([t])), 
\end{eqnarray*}
where $\delta(\cdot)$, $p(\cdot)$ are positive remotely almost-periodic functions,  
and $f(\cdot)$ is a positive function that satisfies a $\gamma$-Lipschitz condition; that is, the following holds:
\begin{eqnarray*}
\vert f(x)-f(y) \vert \leq \gamma \vert x-y\vert,\ \ x,y\in \bb{R}_{+}.
\end{eqnarray*}

We will assume the following condition:  
\begin{enumerate}[label=\textbf{(LW.\arabic*)}]
\item\label{LW1} The average of $\delta$ satisfies $M(\delta) > \delta_{-} > 0$.
\item\label{LW2} The functions $\delta(\cdot)$, $p(\cdot)$ are positive remotely almost-periodic functions,  
and $f(\cdot)$ is a positive function that satisfies a $\gamma$-Lipschitz condition.
\end{enumerate}

The main result is the following theorem:

\begin{thm}\label{gamma_small_unique_sol}
If conditions \ref{LW1} and \ref{LW2} are satisfied, then for sufficiently small $\gamma$,  
the equation \eqref{lasota-wazewska_modificada} has a unique remotely almost-periodic solution.
\end{thm}

Before proving this theorem we need the following technical lemma

\begin{lem}\label{BIRAP_exp_neg}[Bi-Remotelly Integrable Almost Periodicity for an Ergodic function $a$]
If $a$ is a Remotelly Almost Periodic function with $\mathcal{M}(a)\neq 0$, then given $\epsilon>0$ there is $\epsilon'(\epsilon)=\epsilon'>0$, such that if  $\tau\in T(a,\epsilon')$  we have
\begin{eqnarray*}
\limsup_{\vert t \vert \rightarrow \infty} \int_{-\infty}^{t}\left\vert e^{\int_{s+\tau}^{t+\tau}a(r)dr}-e^{\int_{s}^{t}a(r)dr}\right\vert ds < \epsilon ,\ \ \mathcal{M}(a)<0.
\end{eqnarray*}
and
\begin{eqnarray*}
\limsup_{\vert t \vert \rightarrow \infty} \int^{\infty}_{t}\left\vert e^{\int_{s+\tau}^{t+\tau}a(r)dr}-e^{\int_{s}^{t}a(r)dr}\right\vert ds < \epsilon,\ \ , \mathcal{M}(a)>0.
\end{eqnarray*}
\end{lem}

\begin{proof}
Let $\mathcal{M}(a)<-\gamma<0$, we know $\vert e^{\int_{s}^{t}a(r)dr}\vert\leq Ke^{-\gamma(t-s)}$ for $t\geq s$.
It is not hard to verify
\begin{eqnarray*}
\vert e^{\int_{s+\tau}^{t+\tau}a(r)dr}-e^{\int_{s}^{t}a(r)dr}\vert\leq K^{2} (t-s) e^{-\gamma(t-s)}\sup_{u\in (s,t)}\vert a(u+\tau)-a(u)\vert
\end{eqnarray*}
Then
\begin{eqnarray*}
\int_{-\infty}^{t} \vert e^{\int_{s+\tau}^{t+\tau}a(r)dr}-e^{\int_{s}^{t}a(r)dr}\vert
&\leq& \int_{-\infty}^{t} K^{2} (t-s) e^{-\gamma(t-s)}\sup_{u\in (s,t)}\vert a(u+\tau)-a(u)\vert ds\\
&\leq & \int_{0}^{\infty} K^{2} e^{-\gamma x}x\sup_{u\in (t,t+x)}\vert a(u-x+\tau)-a(u-x)\vert ds.
\end{eqnarray*}
Given $\epsilon>0$ consider $\epsilon'=K^{2}\gamma^{-1}\epsilon$. 
Now consider $t\rightarrow\infty$ and let $\{ x_n\}_{n\in\bb{Z}}$ be any sequence suche that $x_n\rightarrow\infty$ when $n\rightarrow\infty$. We have
$$\lim_{t\rightarrow\infty} \sup_{u\in(t,\infty)}\vert a(u-x+\tau)-a(u-x)\vert =\lim_{n\rightarrow\infty} \sup_{u\in(x_n,\infty)}\vert a(u-x+\tau)-a(u-x)\vert.$$
Therefore, by Lebesgue's Dominated Convergence Theorem we get
\begin{eqnarray*}
\lim_{t\rightarrow\infty}\int_{0}^{\infty}e^{-\gamma x}x\sup_{u\in(t,\infty)}\vert \Delta_{\tau} a(u-x)\vert dx 
	&=&	\int_{0}^{\infty}e^{-\gamma x}x\lim_{t\rightarrow\infty}\sup_{u\in(t,\infty)}\vert \Delta_{\tau} a(u-x)\vert dx.
\end{eqnarray*}
Hence
\begin{eqnarray*}
\limsup_{t\rightarrow\infty}
\int_{-\infty}^{t} \vert e^{\int_{s+\tau}^{t+\tau}a(r)dr}-e^{\int_{s}^{t}a(r)dr}\vert ds
&\leq & \limsup_{t\rightarrow\infty} \int_{0}^{\infty} K^{2} e^{-\gamma x}x\sup_{u\in (t,\infty)}\vert \Delta_{\tau} a(u-x)\vert dx\\
&\leq & \epsilon.
\end{eqnarray*}
The case $t\rightarrow -\infty$ is analogous.

The case $\mathcal{M}(a)>0$ concludes in the same way, this finishes the proof.
\end{proof}

We are now ready to prove the aforementioned result. 

\begin{proof}[Proof Theorem \ref{gamma_small_unique_sol}]
Let $\psi$ be a real remotely almost-periodic function, and consider the equation  
\begin{eqnarray}
y'(t) = -\delta(t)\, y(t) + p(t)\, f(\psi([t])). \label{lasota-psi}
\end{eqnarray}

Then, the unique bounded solution of \eqref{lasota-psi} satisfies

\begin{eqnarray*}
y(t)=\int_{-\infty}^{t} e^{ -\int_{u}^{t}\delta (s)ds}  p(u)f(\psi([u]))du.
\end{eqnarray*}

The homogeneous part of equation \eqref{lasota-psi} admits an exponential dichotomy.  
Therefore, by Lemma \ref{BIRAP_exp_neg}, we know that $t\mapsto e^{-\int_{u}^{t} \delta(s)\, ds}$ is bi-remotely almost periodic.  
Then, by applying \cite[Lemma 4]{pinto_remotely}, we conclude the result.
\end{proof}

Taking $f\in BUC(\mathbb{R},\mathbb{R})$ given by $f(x)=e^{ -\gamma x}$, with $\gamma>0$, we have the Lasota-Wazewska Model
\begin{eqnarray}
y'(t)=-\delta(t) y(t)+p(t)e^{-\gamma y([t])},\ t\geq 0. \label{application_depca}
\end{eqnarray}

So we have.

\begin{cor}
If \ref{LW1} and \ref{LW2} are satisfied, then for $\gamma$ small enough, the Lasota-Wazewska model with a piecewise constant argument \eqref{application_depca} has a unique remotely almost periodic solution.
\end{cor}

\printbibliography[heading=bibintoc,title={Bibliography}]

\end{document}